\documentclass[11pt,a4paper,reqno]{article}
\usepackage{amsmath}
\usepackage{graphicx}
\usepackage{comment}
\usepackage{amsfonts}
\usepackage{amssymb}
\usepackage{amsthm}
\usepackage[margin=2.5cm]{geometry}
\usepackage{indentfirst}
\usepackage[all]{xy}
\usepackage[colorlinks=true,linkcolor=cyan]{hyperref}
\usepackage{mathrsfs} 
\usepackage{tikz-cd}
\usetikzlibrary{graphs,decorations.pathmorphing,decorations.markings}
\usepackage{enumitem}
\usepackage[title]{appendix}
\setitemize{noitemsep,topsep=0pt,parsep=0pt,
partopsep=0pt,itemindent=12pt,leftmargin=0pt}
\makeatletter
\newcommand{\subjclass}[2][2020]{%
 \let\@oldtitle\@title%
 \gdef\@title{\@oldtitle\footnotetext{#1 \emph{Mathematics subject classification.} #2}}%
}
\newcommand{\keywords}[1]{%
 \let\@@oldtitle\@title%
 \gdef\@title{\@@oldtitle\footnotetext{\emph{Key words and phrases.} #1.}}%
}
\makeatother


\newtheorem{theorem}{Theorem}[section]
\newtheorem{lemma}[theorem]{Lemma}
\newtheorem{proposition}[theorem]{Proposition}

\theoremstyle{definition}
\newtheorem{definition}[theorem]{Definition}

\theoremstyle{remark}
\newtheorem{remark}[theorem]{Remark}

\newcommand{\build}[3]{\mathrel{\mathop{\kern 0pt#1}\limits_{#2}^{#3}}}

\newcommand\U{{\mathrm U}}
\newcommand\Z{{\mathbb Z}}

\newcommand\R{\mathbb{R}}
\newcommand\C{\mathbb{C}}
\newcommand\Tr{\mathrm{Tr}}
\newcommand\vol{\mathrm{vol}}

\newcommand\Hom{\mathrm{Hom}}

\newcommand\Gbb{{\mathbb{G}}}
\newcommand\Hol{{\mathrm{Hol}}}

\newcommand\mesh{{\mathrm{mesh}}}

\title{Universality of two-dimensional Markovian holonomy fields}

\author{Thibaut Lemoine\thanks{Universit\'e de Strasbourg, CNRS, UMR 7501 -- Institut de Recherche Math\'ematique Avanc\'ee, 7 rue Ren\'e Descartes, 67000 Strasbourg, France. thibaut.lemoine@math.unistra.fr} \hspace{0.1cm}and Elias Nohra\thanks{Sorbonne Universit\'e, Universit\'e Paris Cit\'e, CNRS, Laboratoire de Probabilit\'es, Statistique et Mod\'elisation, LPSM, F-75005, France. Email: elias.nohra@sorbonne-universite.fr}}



\begin{document}

\maketitle

\begin{abstract}
We prove a universality theorem for a broad class of two-dimensional gauge theories on compact surfaces. Each admissible conjugation-invariant L\'evy process on a compact connected Lie group determines a universality class of lattice gauge theories whose continuum limit is the associated Markovian holonomy process. Our result can be seen as a gauge-theoretic analogue of invariance principles for random walks and L\'evy processes.
This framework includes the Yang--Mills holonomy process and the standard heat-kernel (Villain), Wilson, and Manton lattice actions. The proof uses state-sum formulas of independent interest, separating action-dependent spectral coefficients from action-independent topological coefficients determined by the surface and the marked ribbon type of the loop configuration.
\end{abstract}
\tableofcontents

\section{Introduction}
The Yang--Mills holonomy process, introduced by L\'evy in \cite{Lev03,Lev10}, is a stochastic process indexed by loops based at a fixed point $o$ of a compact surface $\Sigma$, valued in a compact Lie group $G$, and satisfying the multiplicativity property expected of holonomies. It provides a rigorous description of the parallel transports of a random connection distributed according to the two-dimensional Yang--Mills measure. As a stochastic process, its parameter is not time but the loop along which the holonomy is computed; correspondingly, its Markov property is spatial and is expressed through cutting and gluing of surfaces \cite[Chapter 5]{Lev03}. 

Building on this analogy, L\'evy constructs in \cite{Lev10} a whole family of two-dimensional Markovian holonomy fields in which the Yang--Mills measure sits exactly as the Brownian motion sits in the class of Lévy processes. In L\'evy's terminology, a Markovian holonomy field is more than a single process on one surface: it is a coherent family of measures on multiplicative holonomy functions, indexed by measured marked surfaces (with possible constraints), and satisfying natural covariance, cutting--gluing, disjoint-union, and Markov-type axioms. After fixing a surface and a base point, and restricting to based loops, such a field induces a holonomy process in the sense used in the present paper. The fields constructed by L\'evy are parametrized by a class of $G$-valued L\'evy processes, called admissible.

This perspective suggests a natural question, which is the focus of the present paper. It is well known that every L\'evy process arises as the scaling limit of a suitable class of discrete random walks. We ask: \emph{what is the gauge-theoretic analogue of a random walk, and do such discrete objects converge, in an appropriate scaling limit, to the holonomy processes induced by Markovian holonomy fields?} In this work, we propose a natural definition of discrete holonomy processes (our ``random walks in gauge theory'') and study their convergence toward the corresponding continuous holonomy processes obtained from L\'evy's fields.

A defining feature of L\'evy's Markovian holonomy fields \cite{Lev03,Lev10} is their exact consistency under subdivision. The convolution-semigroup property implies that subdividing a face and integrating over the newly introduced edge variables leaves the theory unchanged. Generic lattice actions\footnote{In the case of Yang--Mills theory, this includes, for example, the Wilson and Manton actions.}, however, do not enjoy this exact consistency: their discrete holonomy laws genuinely depend on the chosen topological map. Our purpose is to show that this dependence disappears under refinement and to identify precisely which features of the plaquette action survive in the continuum limit.

We prove that, for a broad class of central plaquette weights on an arbitrary compact oriented surface and with arbitrary compact connected structure group, the continuum scaling limit depends only on the infinitesimal Fourier symbol of the plaquette action. More precisely, if the Fourier coefficients of the plaquette weights satisfy
\[
 a_t(\lambda)=\frac{1}{d_\lambda}\langle Q_t,\chi_\lambda\rangle
 =e^{-t\psi(\lambda)}+o_\lambda(t),\qquad t\downarrow0,
\]
for the symbol $\psi$ of an admissible conjugation-invariant L\'evy process, together with a uniform spectral summability condition, then the corresponding lattice holonomy processes converge to the based-loop holonomy process induced by the Markovian holonomy field associated with $\psi$. Consequently, admissible L\'evy symbols parametrize universality classes of two-dimensional lattice gauge theories.

The mechanism behind the theorem is a state-sum factorization separating the action-dependent Fourier data of the plaquette weights from topological coefficients determined by the surface and the loop configuration. Once the marked ribbon type is fixed, these topological coefficients are invariant under refinement, while all dependence on the plaquette action and on the face areas is carried by products of one-plaquette Fourier coefficients. This separation is the basic reason why the infinitesimal symbol $\psi$ alone determines the limit. We describe the factorization precisely below; let us first state our main result.

\subsection{Main result}

Let us fix a compact connected Lie group $G$, and a $\operatorname{Ad}G$-invariant inner product on its Lie algebra $\mathfrak g$. This determines a bi-invariant Riemannian metric and distance on $G$, the Laplace--Beltrami operator, the heat kernel, and the Casimir eigenvalues. Haar measure is always normalized to have total mass one. Let $\Sigma_{g,n}$ be a compact connected oriented surface of genus $g\geq0$ with $n$ boundary components $b_1,\ldots,b_n$. We assume that it is endowed with a smooth positive area density $\vol$, which can for instance be inherited from a Riemannian structure, and we denote by $|D|=\vol(D)$ the area of a domain $D\subset\Sigma_{g,n}$. For any topological map $\Gbb=(V,E,F)$ embedded in $\Sigma_{g,n}$, we consider the space of $G$-valued configurations $G^E$, identified with the space of discrete connections on the trivial $G$-bundle over $\Gbb$. Given a configuration $U=(U_e)_{e\in E}$ and a path $\gamma=e_1^{\varepsilon_1}\cdots e_m^{\varepsilon_m}$, with $\varepsilon_i\in\{\pm1\}$, its holonomy is
\[
 \Hol_\gamma(U)=U_{e_1}^{\varepsilon_1}\cdots U_{e_m}^{\varepsilon_m}.
\]
 The associated unnormalized lattice gauge measure on $\Gbb$, for a family $(Q_t)_{t>0}$ of central continuous probability densities called central plaquette weights and boundary condition $[t_i]$ around the $i$-th boundary component, is
\[
 d\mu(U)=\prod_{f\in F}Q_{|f|}(\Hol_{\partial f}(U))
 \prod_{i=1}^n\delta_{[t_i]}(\Hol_{\partial b_i}(U))\,dU,
\]
where $dU=\prod_{e\in E}dU_e$ and $\delta_{[t_i]}$ is the invariant orbital probability measure on the conjugacy class $[t_i]$ of $t_i\in G$. 

Fix a common base vertex $o$. The gauge-invariant observables are Wilson loop functionals $F(H_{\ell_1},\ldots,H_{\ell_k})$, with $\ell_i\in L_o(\Gbb)$ and $F:G^k\to\R$ invariant under the diagonal adjoint action of $G$, where $H_\ell:U\mapsto\Hol_\ell(U)$. For instance, when $G=\U(N)$, such functionals are generated by products of traces\footnote{See e.g. \cite{Lev04} for a proof.}
\[
 W_{\ell_1,\ldots,\ell_k}=\Tr(H_{\ell_1})\cdots\Tr(H_{\ell_k}).
\]
For a finite based loop family $L=(\ell_1,\ldots,\ell_k)$, we denote the corresponding unnormalized Wilson loop integral by
\[
 \mathcal E_{\Gbb}(F;L)
 :=\int_{G^E}F(\Hol_{\ell_1}(U),\ldots,\Hol_{\ell_k}(U))\,d\mu(U).
\]
When the partition function is non-zero, normalization gives a probability measure, and the full discrete holonomy process can be viewed as a random homomorphism
\[
 H\in\Hom(L_o(\Gbb),G)/G,
\]
since gauge transformations conjugate all based-loop holonomies simultaneously. Its joint holonomy laws are the finite-dimensional marginals; Wilson loop expectations are expectations of conjugation-invariant test functions against these marginals.

The continuous counterpart is the based-loop holonomy process induced by a Markovian holonomy field in the sense of L\'evy \cite{Lev03,Lev10}. Once a surface $\Sigma_{g,n}$ and a base point $o$ are fixed, such a field determines, after normalization, a gauge-invariant law on multiplicative $G$-valued functions of based loops. Its local building blocks are conjugation-invariant L\'evy processes on $G$: Brownian motion yields the ordinary two-dimensional Yang--Mills holonomy process, while more general admissible L\'evy processes give non-Brownian holonomy processes.

For completeness, let us recall the probabilistic input. A stochastic process $X=(X_t)_{t\geq0}$ on $G$, defined on a probability space $(\Omega,\mathcal F,\mathbb P)$, is called a L\'evy process if:
\begin{enumerate}
\item $X_0=1_G$ almost surely;
\item for any $0\leq t_1<t_2<\cdots<t_n$, the increments $X_{t_1}^{-1}X_{t_2},X_{t_2}^{-1}X_{t_3},\ldots,X_{t_{n-1}}^{-1}X_{t_n}$ are independent;
\item for any $0\leq s<t$, the increment $X_s^{-1}X_t$ has the same law as $X_{t-s}$;
\item for every $\varepsilon>0$ and every $t\geq0$,
\[
 \lim_{s\to t}\mathbb P\big(d(X_s,X_t)>\varepsilon\big)=0,
\]
where $d$ is any metric inducing the topology of $G$;
\item the sample paths $t\mapsto X_t$ are almost surely c\`adl\`ag.
\end{enumerate}

Following \cite{Lev10}, a L\'evy process $X$ on $G$ is called \emph{admissible} if it is started at $1_G$ and, for every $t>0$, the law of $X_t$ has a central positive density $Q_t^X$ with respect to Haar measure, such that $(t,g)\mapsto Q_t^X(g)$ is continuous and positive on $\R_+^*\times G$. Centrality implies that the Fourier transform in every irreducible representation $\lambda$ is scalar. The convolution-semigroup property and continuity at $t=0$ therefore give a unique function
\[
 \psi:\widehat G\longrightarrow\C,
 \qquad \Re\psi(\lambda)\geq0,
\]
called the \emph{L\'evy symbol}, such that
\[
 \frac{1}{d_\lambda}\langle Q_t^X,\chi_\lambda\rangle=e^{-t\psi(\lambda)},
 \qquad \lambda\in\widehat G.
\]
Because $Q_t^X$ is real-valued, $\psi(\overline\lambda)=\overline{\psi(\lambda)}$. The symbol is real and nonnegative when the semigroup is invariant under inversion; conjugation invariance alone does not force this additional symmetry. Not all L\'evy processes are admissible: compound Poisson processes, for instance, are not. As discussed in \cite{Lev10}, admissible L\'evy processes nevertheless form the natural class associated with Markovian holonomy fields.

We shall consider more general families $Q=(Q_t)_{t>0}$ that need not themselves form a convolution semigroup; we call them \emph{action weights}. Let $\psi:\widehat G\to\C$ be the L\'evy symbol of an admissible L\'evy process. We say that $Q$ is \emph{scaling-admissible for the symbol $\psi$} if:
\begin{itemize}
\item[(A1)] for every $t>0$, $Q_t$ is a central continuous probability density on $G$;
\item[(A2)] for every $\lambda\in\widehat G$,
\[
 a_t(\lambda):=\frac{1}{d_\lambda}\langle Q_t,\chi_\lambda\rangle
 =e^{-t\psi(\lambda)}+o_\lambda(t),\qquad t\downarrow0;
\]
\item[(A3)] for every $T>0$ and $M\geq0$, there exists $t_0=t_0(T,M)>0$ such that
\begin{equation}\label{eq:single-plaquette-summability}
 \sum_{\lambda\in\widehat G}
 d_\lambda^M
 \sup_{0<t\leq t_0}|a_t(\lambda)|^{T/t}
 <\infty.
\end{equation}
\end{itemize}

Absolute convergence of a single-plaquette character expansion is a separate regularity property. We say that $Q_t$ is \emph{Peter--Weyl regular} if
\begin{equation}\label{eq:PW-sum}
 \sum_{\lambda\in\widehat G}d_\lambda^2|a_t(\lambda)|<\infty.
\end{equation}
Under~\eqref{eq:PW-sum}, the Peter--Weyl series of $Q_t$ converges absolutely and uniformly.

If $X$ is an admissible L\'evy process with symbol $\psi$, then $(Q_t^X)_{t>0}$ is scaling-admissible for $\psi$. Indeed, (A1) and (A2) follow from admissibility and the convolution-semigroup property, while
\[
 \sup_{0<t\leq t_0}|a_t(\lambda)|^{T/t}=e^{-T\Re\psi(\lambda)},
\]
which is summable in  $\lambda$ with the weight $d_\lambda^M$ by Lemma~\ref{lem:levy-symbol-summability}. In addition, every $Q_t^X$ is Peter--Weyl regular: the semigroup property gives $a_t(\lambda)=a_{t/2}(\lambda)^2$, and Plancherel yields
\[
 \sum_{\lambda\in\widehat G}d_\lambda^2|a_t(\lambda)|
 =\sum_{\lambda\in\widehat G}d_\lambda^2|a_{t/2}(\lambda)|^2
 =\|Q_{t/2}^X\|_{L^2(G)}^2<\infty.
\]
The point of scaling-admissibility is that the semigroup identity itself is unnecessary: only its infinitesimal Fourier behaviour, supplemented by the uniform summability in (A3), survives in the scaling limit.

There remains a geometric issue: the discrete process on $\Gbb_n$ and the limiting process on $\Sigma$ have different index sets. We handle this through regular loop families, stable approximations, and universally stable sequences of topological maps, leading to the notion of regular convergence introduced in Definition~\ref{def:regular-process-convergence}. The geometric framework is described in detail below in the subsection on geometric input and proof strategy.

\begin{theorem}[Universality of holonomy processes]\label{thm:universality}
Let $\Sigma$ be a compact connected oriented Riemannian surface, possibly with boundary, endowed with a smooth positive area density $\vol$, and fix a base point $o\in\operatorname{int}(\Sigma)$. Let $G$ be a compact connected Lie group, let $(\Gbb_n)_{n\geq1}$ be a universally stable sequence of embedded topological maps on $\Sigma$ based at $o$, and let $(Q_t)_{t>0}$ be scaling-admissible with L\'evy symbol $\psi$. In the case with boundary, fix the same prescribed boundary conjugacy classes $\mathfrak t$ for the lattice models and for the limiting Markovian holonomy field.
\par Then the lattice partition function is strictly positive for all sufficiently large $n$. For those $n$, let $H_n$ be the holonomy process on $\Gbb_n$ associated with the plaquette weights $(Q_t)_{t>0}$ and boundary data $\mathfrak t$ when present. Then $(H_n)$ converges regularly to the holonomy process $H^\psi$ on $\Sigma$ induced by the regular Markovian holonomy field associated with the L\'evy process of symbol $\psi$ and the same boundary data.
\end{theorem}
\begin{remark}
The qualification ``regularly'' is necessary because $H_n$ and $H^\psi$ have different index sets: $H_n$ is indexed by $L_o(\mathbb G_n)$, whereas $H^\psi$ is indexed by rectifiable loops in $\Sigma$. Definition~\ref{def:regular-process-convergence} simultaneously controls the approximation of loops and the convergence of their holonomies.
\end{remark}
The theorem identifies an entire universality class with each admissible L\'evy symbol. The most important example is the Brownian class,
\[
 \psi_{\mathrm{YM}}(\lambda)=\frac12c_2(\lambda),
\]
where $c_2(\lambda)$ is the Casimir eigenvalue of the irreducible representation $\lambda$. The limiting process is then the ordinary two-dimensional Yang--Mills holonomy process. This class contains the heat-kernel action, but also the Wilson and Manton actions. Thus these microscopically different lattice theories have the same holonomy-level continuum limit.

The statement is not restricted to diffusions: if $f$ is a Bernstein function such that the subordinated heat kernels on $G$ are admissible, then Brownian motion subordinated by the subordinator with Laplace exponent $f$ has symbol
\[
 \psi_f(\lambda)=f\!\left(\frac12c_2(\lambda)\right).
\]
For example, $f(u)=u^\alpha$, $0<\alpha<1$, gives a stable-type non-Brownian Markovian holonomy field. The theorem identifies the scaling limit of any scaling-admissible plaquette action with infinitesimal symbol $\psi_f$ with the holonomy process induced by this non-Brownian field. In particular, the holonomy formulation treats Brownian and genuinely non-Brownian universality classes within the same framework.

\subsection{State-sum factorization and the mechanism of universality}

The analytic content of the scaling assumption becomes effective because unnormalized Wilson loop integrals admit state-sum expansions in which all dependence on the action and on the face areas is separated from the topology of the loop configuration. These are the first structural results of the paper. For $t>0$ and $\lambda\in\widehat G$, write as above $a_t(\lambda)=\frac{1}{d_\lambda}\langle Q_t,\chi_\lambda\rangle$.

\begin{theorem}[Partition function]\label{thm:Fourier_PF}
Let $\Sigma_{g,n}$ be a compact connected oriented surface and let $\Gbb=(V,E,F)$ be a topological map embedded in $\Sigma_{g,n}$, with every boundary component carried by $\Gbb$. Let $(Q_t)$ be a family of central continuous probability densities and assume the graph-level spectral summability condition
\begin{equation}\label{eq:PF-spectral-summability}
 \sum_{\lambda\in\widehat G}
 d_\lambda^{2-2g}
 \prod_{f\in F}|a_{|f|}(\lambda)|<\infty.
\end{equation}
Then the partition function has the absolutely convergent expansion
\begin{equation}\label{eq:pf-state-sum-intro}
 Z_{\Gbb,G,Q,\vol}(t_1,\ldots,t_n)
 =\sum_{\lambda\in\widehat G}
 \left(\prod_{f\in F}a_{|f|}(\lambda)\right)
 d_\lambda^{2-2g}\prod_{j=1}^n\frac{\chi_\lambda(t_j)}{d_\lambda}.
\end{equation}
\end{theorem}

Condition~\eqref{eq:PF-spectral-summability} concerns the complete product of face coefficients rather than the character expansion of any individual plaquette. It is automatically satisfied if one face weight is Peter--Weyl regular, since $|a_t(\lambda)|\leq1$. More importantly for scaling limits, assumption (A3) implies it when the faces are sufficiently small while their total area stays bounded below.

Formula~\eqref{eq:pf-state-sum-intro} is classical for the heat-kernel action \cite{Rus90,Wit91,BlauThom92}, where subdivision invariance allows one to reduce directly to a one-face graph. It is stated without proof at the same level of generality in \cite{DangNohra26}, and it may also be compared with the spin-foam formulas obtained in the physics literature \cite{OeckPfei01}. We give a complete proof for general central actions. The important feature for the present paper is that the factors depending on the plaquette weights and on the areas are cleanly separated from the topological factor $d_\lambda^{2-2g}\prod_j\chi_\lambda(t_j)/d_\lambda$.

The same principle extends to arbitrary gauge-invariant functions of finitely many holonomies. Fix a vertex $o$ and let $L=(\ell_1,\ldots,\ell_k)$ be a regular\footnote{See Definition~\ref{def:regular-loop-family}.} finite family of graph loops in $L_o(\Gbb)$, and let $\Gamma(L)$ be the finite ribbon subgraph formed by the edges visited by the loops. Cutting $\Sigma_{g,n}$ along $\Gamma(L)$ gives compact surfaces with boundary $C_1,\ldots,C_r$; let $F_i$ be the set of faces contained in $C_i$. Choose one orientation of each cut edge and denote the corresponding set by $E_L$. If $U=(U_e)_{e\in E_L}\in G^{E_L}$, write $h_{\ell_i}(U)$ for the word corresponding to $\ell_i$. For each component $C_\alpha$, denote its boundary holonomies by
\[
 x_{\alpha,1}(U),\ldots,x_{\alpha,m_\alpha}(U).
\]
Each of them is either one of the prescribed boundary holonomies $t_j$, up to conjugacy, or a word in the variables $U_e^{\pm1}$. For any configuration $\Lambda:\{C_1,\ldots,C_r\}\to\widehat G$, define the \emph{spectral coefficient}
\begin{equation}\label{eq:spec-coef}
 \kappa_{\Gbb,Q,\vol}(\Lambda)
 =\prod_{i=1}^r\left(\prod_{f\in F_i}a_{|f|}\big(\Lambda(C_i)\big)\right)
\end{equation}
and the \emph{topological coefficient}
\begin{equation}\label{eq:top-coef}
 \widehat W_{F,L}(\Lambda)
 =\int_{G^{E_L}}F\big(h_{\ell_1}(U),\ldots,h_{\ell_k}(U)\big)
 \prod_{i=1}^r\left[d_{\Lambda(C_i)}^{\chi(C_i)}
 \prod_{j=1}^{m_i}\chi_{\Lambda(C_i)}\big(x_{i,j}(U)\big)\right]dU.
\end{equation}

\begin{theorem}[Wilson loop integrals]\label{thm:Fourier_Wilsonloop}
Let $\Sigma_{g,n}$ be a compact surface of genus $g$ with $n$ boundary components, with prescribed boundary conjugacy classes $\mathfrak t=([t_1],\ldots,[t_n])$ when $n>0$. Let $\Gbb=(V,E,F)$ be a topological map embedded in $\Sigma_{g,n}$, fix a vertex $o\in V$, and let $L=(\ell_1,\ldots,\ell_k)$ be a finite regular family of loops in $L_o(\Gbb)$. Denote by $C_1,\ldots,C_r$ the surfaces with boundary obtained by cutting $\Sigma_{g,n}$ along $\ell_1,\ldots,\ell_k$, and let $F_\alpha$ be the set of faces contained in $C_\alpha$. Let $(Q_t)$ be a family of central continuous probability densities and assume that, for every $1\leq\alpha\leq r$,
\begin{equation}\label{eq:component-spectral-summability}
 \sum_{\lambda\in\widehat G}
 d_\lambda^{2-2g_\alpha}
 \prod_{f\in F_\alpha}|a_{|f|}(\lambda)|<\infty,
\end{equation}
where $g_\alpha$ is the genus of $C_\alpha$. For any $F\in C(G^k)^{\operatorname{Ad}G}$, set
\[
 \mathcal E_{\Gbb}(F;L)
 :=\int_{\Omega^1(\Gbb,G)}
 F\bigl(\Hol_{\ell_1}(U),\ldots,\Hol_{\ell_k}(U)\bigr)\,
 d\mu_{\Gbb,G,Q,\vol,\mathfrak t}(U).
\]
Then we have the absolutely convergent decomposition
\begin{equation}\label{eq:decomp_spectrale_WL}
 \mathcal E_{\Gbb}(F;L)
 =\sum_{\Lambda}\kappa_{\Gbb,Q,\vol}(\Lambda)\widehat W_{F,L}(\Lambda).
\end{equation}
\end{theorem}

The terminology is deliberate. The coefficient $\kappa_{\Gbb,Q,\vol}(\Lambda)$ is a product of Fourier--Plancherel coefficients of the plaquette weights and contains all dependence on the action, the areas, and the refinement. By contrast, $\widehat W_{F,L}(\Lambda)$ depends only on the topology of the surface and of the loop configuration and not on the volume measure, the plaquette action, or the ambient graph once the marked ribbon type is fixed.

This makes the universality mechanism transparent. For a fixed component $C$ and a fixed representation $\lambda$, a stable refinement gives, schematically,
\[
 \prod_{f\subset C}a_{|f|}(\lambda)
 =\prod_{f\subset C}\big(e^{-|f|\psi(\lambda)}+o_\lambda(|f|)\big)
 \longrightarrow e^{-|C|\psi(\lambda)}.
\]
The topology-dependent coefficient does not change along the stable approximation. Assumption (A3) supplies the uniform spectral domination needed to pass this convergence through the representation sums. Thus the exact state-sum formula turns the scaling problem into a product limit for one-plaquette Fourier coefficients, which is the representation-theoretic core of Theorem~\ref{thm:universality}.

These state-sum formulas are much easier for the heat-kernel action because the discrete two-dimensional Yang--Mills measure is exactly invariant under subdivision. For a general action there is no such reduction, and part of the work is to develop a systematic graph-independent integration procedure. Theorem~\ref{thm:Fourier_Wilsonloop} is reminiscent of conformal bootstrap decompositions in CFT or of IRF models. For $G=\U(N)$, L\'evy recently obtained a purely combinatorial explicit formula \cite{Lev26} when $F$ is a product of traces; his results are stated for the heat-kernel action, although many of the arguments apply more generally.

For the scaling-limit theorem we do \emph{not} need an explicit expression for every $\widehat W_{F,L}(\Lambda)$: its stability and action-independence are enough. Explicit control becomes important in other asymptotic regimes. For example, Dahlqvist \cite{Dah26} computed large-$N$ limits of Wilson loops for the heat-kernel action using a refined extension of Koike--Schur--Weyl duality. Related representation-theoretic state-sum ideas extend to higher-dimensional lattice gauge theory and are developed for $\U(N)$ with general actions by the first author in \cite{Lem26b}, with an application to the heat-kernel action at strong coupling in \cite{Lem26c}.
\par We also mention that it would be of interest to see if such state-sum formulas can be applied to the study of the zero--area limit of the Markovian holonomy fields, beyond the Yang--Mills case (see for example \cite{Sen97limit,Liu96,DangNohra26b}).

\subsection{Relation with previous results}

The continuum limit of two-dimensional lattice gauge theories has been approached from several complementary directions. The exact character formulas of Rusakov, Witten, and Blau--Thompson \cite{Rus90,Wit91,BlauThom92} exhibit the special solvability of the two-dimensional theory. Driver \cite{Dri89} proved convergence of lattice Yang--Mills theories at the level of expectations of gauge-invariant observables associated with finitely many curves, thereby identifying the continuum theory through its finite-dimensional holonomy distributions. Gross--King--Sengupta \cite{GKS89} constructed the continuous Yang--Mills measure directly through stochastic parallel transport, and Sengupta extended this construction to compact surfaces \cite{Sen97}.

A structural construction at the level of holonomies was subsequently developed by L\'evy \cite{Lev03,Lev10}. He first constructed the Yang--Mills holonomy process and then general Markovian holonomy fields from consistent families of finite-dimensional distributions satisfying a Markov property. This is the continuum framework in which our limiting objects live. The present result runs in the opposite direction: starting from microscopic lattice models that are generally not projectively consistent under subdivision, it identifies the based-loop holonomy processes induced by L\'evy's fields as their universal scaling limits.

A different line of work realizes the Yang--Mills measure as a random distributional connection. This involves a highly technical study of this measure in Banach spaces of distributional connections, with regularity analysis using Sobolev or H\"older--Besov scales. On the torus, Chevyrev \cite{Chev19} constructed the Yang--Mills measure as a random distribution with good regularity properties, enabling the application of SPDE techniques and regularity structures to its study. Indeed, Chandra--Chevyrev--Hairer--Shen \cite{CCHS22} constructed the associated Langevin dynamics, and subsequent work proved universality of the Yang--Mills measure for broad classes of lattice approximations; see Chevyrev--Shen \cite{ChevShen26}. Recently, constructions and universality results on general surfaces were obtained using the Morse gauge by Dang and the second author \cite{DangNohra26}, while Chevyrev--Klose--Mohammed \cite{CKM} developed a PDE approach on the square, and Bonthonneau--Chhaibi--Dang--Rivi\`ere--T\^o \cite{BCDRT26} a dynamical-systems approach directly in the continuum to the Yang--Mills measure on compact surfaces.

These results and the present theorem address complementary levels of description. The distributional-connection framework provides a strong functional-analytic treatment of the Brownian class, establishing the scaling limit in suitable Banach spaces of distributional connections that capture the analytic regularity of the Yang--Mills random connection. Our aim, by contrast, is to obtain a universality statement directly for holonomy processes and to generalize it to the full admissible Lévy–Markovian class, including genuinely non-Brownian symbols\footnote{For instance, subordinated or jump-type Markovian holonomy fields, for which the Brownian Yang--Mills measure need not be the natural limiting object.}. It is closer in spirit to  Chevyrev--Garban \cite{ChevGar25}.

Nevertheless, it would be of interest to investigate whether techniques from \cite{Chev19,DangNohra26,CKM,BCDRT26} can be applied to Markovian holonomy fields beyond the Yang--Mills case, in order to realize them as random distributional connections and to analyze their regularity and scaling limits in Banach spaces of distributions, depending on the Lévy symbol.

\subsection{Geometric input and proof strategy}

The remaining difficulty is geometric rather than spectral. A discrete holonomy process on $\Gbb_n$ is naturally indexed by the group $L_o(\Gbb_n)$ of reduced based loops, while the holonomy process induced on $\Sigma$ by a continuum Markovian holonomy field is indexed by rectifiable based loops in $\Sigma$. Since the index sets vary with $n$, ordinary finite-dimensional convergence is not by itself a complete formulation of the scaling limit.

We resolve this by introducing \emph{regular loop families}, Definition~\ref{def:regular-loop-family}, together with stable graph approximations preserving their marked ambient ribbon type and the areas of the complementary components. Convergence along all such approximations leads to the notion of \emph{regular convergence}, Definition~\ref{def:regular-process-convergence}; the limiting process then extends from regular loops to all rectifiable loops by stochastic continuity. To make this simultaneous for every finite regular family, we introduce \emph{universally stable} graph sequences, Definition~\ref{def:universally-stable-graph-sequence}, and prove in Proposition~\ref{prop:existence-universally-stable-sequence} that such sequences exist on every compact smooth surface.

With these notions in place, the proof of Theorem~\ref{thm:universality} separates into two independent ingredients. First, Theorem~\ref{thm:Fourier_Wilsonloop}, together with scaling-admissibility, gives convergence for every fixed regular loop family along every stable graph approximation. Second, the deterministic universal-stability construction provides a single vanishing-mesh graph sequence that contains stable approximations of all finite regular loop families. This separation mirrors the state-sum formula itself: the first step is spectral and probabilistic, while the second is geometric and deterministic.

The paper is organized accordingly. Section~2 reviews the discrete gauge-theoretic and representation-theoretic preliminaries. Section~3 proves the partition-function and Wilson-loop state-sum formulas. Section~4 introduces regular loop families, stable approximations, universally stable graph sequences, and regular convergence, and proves the universality theorem from the state-sum expansion. Section~5 proves the existence of universally stable graph sequences.
\section*{Acknowledgments}
\par E.N. is deeply grateful to his PhD advisor, Thierry Lévy, and to his collaborator, Nguyen Viet Dang, for their constant support, sharp insights, and kindness. He also thanks Nguyen Viet Dang and Nicolas Fournier for valuable discussions on Lévy processes, and Thomas Duquesne for his wonderful course on Markov processes at Sorbonne Université. 
Finally, he gratefully acknowledges the financial support of the ``Fondation CFM pour la Recherche'' and thanks the foundation for providing excellent working conditions during his PhD. 
\par T.L. thanks Nguyen Viet Dang for stimulating discussions about state-sum formalism and approximation theory.

\section{Preliminaries on discrete gauge theory}

In this section we review the basic definitions of lattice gauge theory, and we introduce a few useful results.

\subsection{Graphs, paths and loops}

In this paper, we will repeatedly go back and forth between combinatorial properties of graphs and topological properties of surfaces, because lattice gauge theory lives on the former and the holonomy processes induced by Markovian holonomy fields live on the latter. For this reason, we recall a few definitions. The first one is that of a topological map.

\begin{definition}\label{def:topological_map}
Let $\Sigma_{g,n}$ be a compact surface of genus $g$ with $n$ connected boundary components. A \emph{topological map} embedded in $\Sigma_{g,n}$ is a cellular embedding $\Gbb=(V,E,F)$ of a graph $(V,E)$. 
\end{definition}

Concretely, it means the following: start with a combinatorial graph $\tilde{G}=(\tilde{V},\tilde{E})$, together with embeddings $\theta_V:\tilde{V}\to\Sigma_{g,n}$ and $\theta_E:\tilde{E}\to\Sigma$ such that the images of vertices are distinct points of the surface, and the images of edges are continuous curves on the surface, that only meet at their endpoints, which are images of vertices. We denote by $V=\theta_V(\tilde{V})$ and $E=\theta_E(\tilde{E})$ the images of vertices and edges respectively (note that $(V,E)$ is also a combinatorial graph if we forget the topological structure). Finally, $F$ is the set of faces, which are the connected components of $\Sigma_{g,n}\setminus(V\cup E)$. Vertices (resp. edges, faces) are 0-cells (resp. 1-cells, 2-cells) of the topological map. An important assumption is that all faces are simply connected, \emph{i.e.}, homeomorphic to a disk. A related notion is that of a ribbon graph.

\begin{definition}A ribbon graph is a finite graph $\Gamma=(V,E)$, possibly with loops and multiple edges, together with, for every vertex $v\in V$, a cyclic ordering of the half-edges incident to $v$. An isomorphism of ribbon graphs is a graph isomorphism preserving these cyclic orders.
\end{definition}

If $\Sigma_{g,n}$ is an oriented surface of genus $g$ with $n$ connected boundary components, every finite graph embedded in $\Sigma_{g,n}$ with edges meeting only at their endpoints carries a canonical ribbon structure: the cyclic order at each vertex is the cyclic order induced by the orientation of $\Sigma$. We call such an object an embedded ribbon graph. An embedded ribbon graph is not required to be cellular. However, if the complementary components are disks, then the embedded ribbon graph determines a topological map. Figures~\ref{fig:ribbon} and~\ref{fig:ribbon-map} illustrate the difference between a ribbon graph and a topological map: in Figure~\ref{fig:ribbon}, we have a graph $\Gbb$ embedded in $\Sigma_{3,1}$. The connected components of $\Sigma_{3,1}\setminus\Gbb$ are not homeomorphic to disks, therefore the graph embedding is not cellular, and the graph is only a ribbon graph. In order to obtain a topological map, additional edges must be given, which results for instance in the new graph in Figure~\ref{fig:ribbon-map}.

\begin{figure}[h!]
 \centering
 \includegraphics[width=0.7\linewidth]{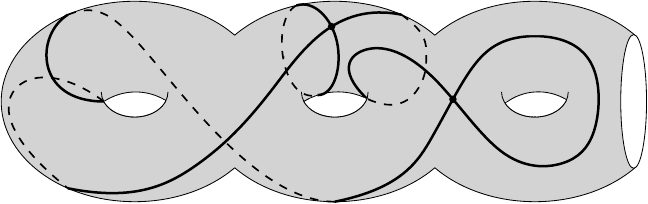}
 \caption{A ribbon graph embedded in $\Sigma_{3,1}$ with two vertices and four edges.}
 \label{fig:ribbon}
\end{figure}

As $\Sigma_{g,n}$ is oriented, the faces have a natural orientation coming from their embedding. Each edge $e\in E$ can be endowed with an orientation so that it has a starting point $\underline{e}$ and an endpoint $\overline{e}$. We assume that an initial orientation of all edges is fixed, and we allow the following action on the graph: for any $e\in E$, we denote by $e^{-1}$ the curve obtained by reversing the parametrization of $e$. It is formally the same edge, but its starting point and endpoint are exchanged. It yields an action of $\{\pm1\}\cong\Z_2$ on $E$ given by $1\cdot e=e^1=e$ and $(-1)\cdot e=e^{-1}$. In the following, when $e_1,\ldots,e_k\in E$ will be fixed, we will allow implicitly to consider them with their initial orientation $e_i$ or with their reverse orientation $e_i^{-1}$.

\begin{figure}[h!]
 \centering
 \includegraphics[width=0.7\linewidth]{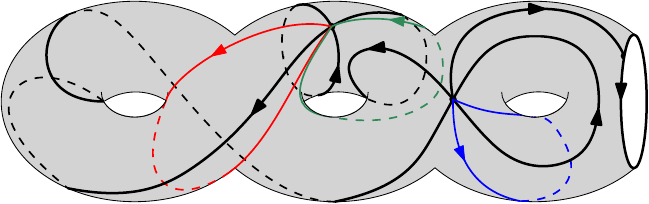}
 \caption{An oriented topological map that contains the ribbon graph from Figure~\ref{fig:ribbon}.}
 \label{fig:ribbon-map}
\end{figure}

\begin{definition}
A \emph{curve} in a topological map $\Gbb=(V,E,F)$ is either a single vertex, or a concatenation $e_1\ldots e_k$ of edges $e_1,\ldots,e_k\in E$, where $\overline{e_i}=\underline{e_{i+1}}$ for all $i$. Given a curve $\gamma$, we define its starting point $\underline{\gamma}$ and endpoint $\overline{\gamma}$ as follows:
\begin{enumerate}
\item If $\gamma=v$ is a vertex, then $\underline{\gamma}=\overline{\gamma}=v$.
\item If $\gamma=e_1\ldots e_k$, then $\underline{\gamma}=\underline{e_1}$ and $\overline{\gamma}=\overline{e_k}$.
\end{enumerate}
\end{definition}

In the definition above, $e_1,\ldots,e_n$ are not necessarily distinct, and they can be endowed with any orientation. A convenient way to encode a non-constant curve $\gamma$ is as an element of the free group $\mathbb{F}_E$ generated by the set of edges. The curve is thus written as a word $w$ in all edges and their inverse.

A curve $\gamma$ such that $\underline{\gamma}=\overline{\gamma}=v$ is called a \emph{loop} with base $v$. Another important assumption in the definition of the map $\Gbb$ is that all boundary components of $\Sigma_{g,n}$ are loops in $\Gbb$. We denote by $L(\Gbb)$ the set of loops in $\Gbb$. For a fixed vertex $v$, we identify based edge loops that differ by insertion or deletion of an immediate backtracking $ee^{-1}$; concatenation followed by reduction makes the resulting set $L_v(\Gbb)$ into a group, the \emph{reduced based edge-path group}. Its identity is the constant loop at $v$, and inversion is given by reversal of the edge word.

\subsection{Configuration space of lattice gauge theory}

For any compact surface $\Sigma$ and any topological map $\Gbb=(V,E,F)$ embedded in $\Sigma$, we build a configuration space of discrete gauge theory as follows:
\begin{enumerate}
\item A discrete connection is given by a function $U:E\to G$ such that $U(e^{-1})=U(e)^{-1}$, or equivalently, a configuration $U=(U_e)_{e\in E}\in G^E$ that is compatible with the change of orientation.
\item The holonomy of a discrete connection $U$ along a path $\gamma=e_1\ldots e_n$ in $\Gbb$ is defined by
\[
\Hol_\gamma(U)=U_{e_1}\ldots U_{e_n}.
\]
The curvature of a discrete connection $U$ is the function $\Omega^U:F\to G$ defined (up to conjugation) by
\[
\Omega^U(f)=\Hol_{\partial f}(U),
\]
for any face $f\in F$.
\end{enumerate}

\begin{remark}[Holonomies and words]\label{rmk:holo}
An important caveat is the following. Holonomy is often defined covariantly by
\[
 \Hol_\gamma(U)=U_{e_n}\cdots U_{e_1},
\]
so that for two compatible curves $\gamma_1,\gamma_2$ one has
$\Hol_{\gamma_1\gamma_2}(U)=\Hol_{\gamma_2}(U)\Hol_{\gamma_1}(U)$: functions are composed from right to left, whereas paths are concatenated from left to right. With our convention, fixing a discrete connection $U$ defines instead a word-evaluation homomorphism
\[
 \operatorname{ev}_U:\mathbb F_E\longrightarrow G,
 \qquad
 w\longmapsto w(U_e:e\in E),
\]
and, whenever the word $w$ represents a path $\gamma$, one has
\begin{equation}\label{eq:hol_words}
 \Hol_\gamma(U)=\operatorname{ev}_U(w)=w(U_e:e\in E).
\end{equation}
Under the usual covariant convention, the corresponding word-evaluation rule is an antihomomorphism. Note that not every element of the free group $\mathbb F_E$ is represented by a path, because successive edge letters must be composable. Our convention is therefore only a convenient way of evaluating the words that do represent paths.
\end{remark}

The set $\Omega^1(\Gbb,G)\simeq G^{\vert E\vert}$ of discrete connections is a compact Lie group, hence it possesses a Haar probability measure $dU$. The gauge group $G^V$ acts on $\Omega^1(\Gbb,G)$ by
\[
(j\cdot U)_e = j(\underline{e})^{-1}U_e j(\overline{e}),\qquad j:V\to G.
\]
With our convention for holonomy, this gives
\[
\Hol_\gamma(j\cdot U)=j(\underline{\gamma})^{-1}\Hol_\gamma(U)j(\overline{\gamma}).
\]
In particular, based loop holonomies are conjugated by the gauge variable at the base point. For this reason, gauge invariance of discrete connections translates to a conjugation invariance of holonomies.

\subsection{Discrete Yang--Mills measure}

\begin{definition}\label{def:discrete-YM-measure}
Let $\Sigma$ be a compact surface of genus $g$ with boundary components
$b_1,\ldots,b_n$, endowed with an area measure $\vol$. Let
$\Gbb=(V,E,F)$ be a topological map embedded in $\Sigma$, let $G$ be a
compact group, and let $(Q_t)_{t>0}$ be a family of central continuous
probability densities on $G$. Fix conjugacy classes
$[t_1],\ldots,[t_n]$ and write
$\mathfrak t=([t_1],\ldots,[t_n])$.

For every $i$, choose an oriented edge $e_i$ contained in the $i$-th
boundary component; these edges are pairwise distinct. Once all variables
except $U_{e_i}$ are fixed, the boundary holonomy has the form
\[
 \Hol_{\partial b_i}(U)=a_i(U)U_{e_i}^{\varepsilon_i}b_i(U),
 \qquad \varepsilon_i\in\{\pm1\},
\]
so, for every $y_i\in G$, there is a unique value of $U_{e_i}$ for which
\[
 \Hol_{\partial b_i}(U)=y_it_iy_i^{-1}.
\]
Denote by $U^{\mathbf y}$ the resulting connection, where
$\mathbf y=(y_1,\ldots,y_n)$. The \emph{discrete Yang--Mills measure}
with weight $Q$ and boundary condition $\mathfrak t$ is the finite measure
characterized by
\begin{equation}\label{eq:DS-rigorous}
\begin{aligned}
 \int_{\Omega^1(\Gbb,G)}\! \Phi(U)\,
 d\mu_{\Gbb,G,Q,\vol,\mathfrak t}(U)
 :=
 \int_{G^{E\setminus\{e_1,\ldots,e_n\}}}\int_{G^n}
 &\Phi(U^{\mathbf y})
 \prod_{f\in F}Q_{|f|}\bigl(\Hol_{\partial f}(U^{\mathbf y})\bigr)
 \,d\mathbf y\,dU
\end{aligned}
\end{equation}
for every bounded measurable $\Phi$. This definition is independent of the
chosen boundary edges, by invariance of Haar measure under left and right
translations and inversion.
\end{definition}

We use the customary shorthand
\begin{equation}\label{eq:DS}
 d\mu_{\Gbb,G,Q,\vol,\mathfrak t}(U)
 =
 \prod_{f\in F}Q_{|f|}(\Hol_{\partial f}(U))
 \prod_{i=1}^n\delta_{[t_i]}(\Hol_{\partial b_i}(U))\,dU,
\end{equation}
where $\delta_{[t]}$ denotes invariant orbital probability measure:
\begin{equation}\label{eq:Dirac}
 \int_G F(x)\,\delta_{[t]}(dx)
 =\int_G F(yty^{-1})\,dy.
\end{equation}
Formula~\eqref{eq:DS} is always interpreted through
\eqref{eq:DS-rigorous}; it is not a product of distributions on
$\Omega^1(\Gbb,G)$.

In the case of the heat kernel action, Equation~\eqref{eq:DS} is known as the Driver--Sengupta equation, named after Driver \cite{Dri89} and Sengupta \cite{Sen97} who used it to construct the continuous two-dimensional Yang--Mills measure. Note that the term inside the action is exactly the curvature $\Omega^U(f)$ for each face. By Remark~\ref{rmk:holo}, we can replace the holonomies along all boundaries (of faces and of the surface) by words in the variables $U_e,e\in E$. The measure $\mu_{\Gbb,G,Q,\vol,\mathfrak t}$ is finite because every
$Q_{|f|}$ is continuous on the compact group $G$. Its mass is the
\emph{partition function}
\[
 Z_{\Gbb,G,Q,\vol}(t_1,\ldots,t_n)
 =\int_{\Omega^1(\Gbb,G)}1\,
 d\mu_{\Gbb,G,Q,\vol,\mathfrak t}.
\]
If $E=\{e_1,\ldots,e_k\}$ and $F=\{f_1,\ldots,f_p\}$, we will use the
formal but convenient notation
\begin{equation}\label{eq:YM_distrib}
\begin{aligned}
 \int \Phi(U)\,d\mu(U)
 =\int_{G^k}
 &\Phi(U_{e_1},\ldots,U_{e_k})
 \prod_{i=1}^p Q_{|f_i|}\bigl(w_i(U)\bigr)
 \prod_{j=1}^n\delta_{[t_j]}\bigl(w'_j(U)\bigr)\,dU,
\end{aligned}
\end{equation}
where $w_i$ and $w'_j$ are the face and boundary words. As above,
\eqref{eq:YM_distrib} is shorthand for the rigorous iterated integral
\eqref{eq:DS-rigorous}.
\par In two-dimensional lattice Yang--Mills theory, the heat-kernel (or Villain) action is sometimes replaced by either the Wilson or the Manton action. It is classical\footnote{See \cite{ChevGar25,ChevShen26,DangNohra26} for similar statements; our claim follows from similar arguments.} that both actions satisfy assumptions (A1)--(A3) with the same symbol as the heat kernel, and we therefore omit the details. Consequently, our results imply that, at the level of holonomy processes, the Yang--Mills measure arises as the scaling limit of lattice gauge theories defined by the Wilson and Manton actions.
\par In the case of an action induced by another semigroup, Equation~\eqref{eq:DS} gives a finite-dimensional distribution of the holonomy process induced by the corresponding Markovian holonomy field constructed by L\'evy in~\cite{Lev10}.

\subsection{Edge integration}

\begin{definition}
Let $\Gbb=(V,E,F)$ and $\Gbb'=(V',E',F')$ be two topological maps embedded in the same surface. $\Gbb$ is said to be \emph{finer} than $\Gbb'$ if all edges of $\Gbb'$ can be written as paths in $\Gbb$. In this case we write $\Gbb'\preceq\Gbb$.
\end{definition}

If $\Gbb'\preceq\Gbb$, there is a restriction map $\mathcal R:\Omega^1(\Gbb,G)\to\Omega^1(\Gbb',G)$. For each oriented coarse edge $e\in E'$, choose its expression $e=e_1\cdots e_k$ as a path of oriented edges of $\Gbb$ and set $\mathcal R(U)_e=U_{e_1}\cdots U_{e_k}$. This is independent of the chosen initial orientations, since reversing an edge replaces its variable by its inverse.

Let us state a fact that is simple yet important for the sequel. Recall that the degree of a vertex $v\in V$ of a graph $(V,E)$ is the number of edges that are adjacent to $v$.

\begin{proposition}[Edge merging]\label{prop:edge_fusion}
Let $\Gbb=(V,E,F)$ be a topological map embedded in a surface $\Sigma_{g,n}$ that contains a subdivision vertex $v\in V$: the two half-edges incident to $v$ belong to two distinct non-loop edge segments $e_1$ and $e_2$, and no other half-edge is incident to $v$. Let $\Gbb'=(V',E',F')$ be a new topological map obtained by removing $v$ and replacing $\{e_1,e_2\}$ by a single edge $e_{12}$ whose embedding is given by the same curve as $e_1e_2$. Then
\begin{equation}
\mathcal{R}_*\mu_{\Gbb,G,Q,\vol,\mathfrak{t}}=\mu_{\Gbb',G,Q,\vol,\mathfrak{t}}.
\end{equation}
\end{proposition}

\begin{proof}
We need to prove that for any continuous test function $F:\Omega^1(\Gbb',G)\to\C$,
\[
\int_{\Omega^1(\Gbb,G)}F(\mathcal{R}(U))d\mu_{\Gbb,G,Q,\vol,\mathfrak{t}}(U)=\int_{\Omega^1(\Gbb',G)}F(U)d\mu_{\Gbb',G,Q,\vol,\mathfrak{t}}(U).
\]
Let $E=\{e_1,e_2,\ldots,e_k\}$, with the convention that $e_1$ and $e_2$ are adjacent to $v$ as in the assumption. Let us also assume that they are oriented so that $v=\overline{e_1}=\underline{e_2}$. We have $\mathcal{R}(U)_{e_{12}}=U_{e_1}U_{e_2}$, and $\mathcal{R}(U)_{e_i}=U_{e_i}$ for all $i\geq 3$. For all $f\in F$, let $w_f$ be the word associated to $\partial f$, and for all $1\leq i\leq n$, let $w'_i$ be the word associated to $\partial b_i$. The important fact is that every occurrence of these edge germs in a face or boundary word is through $e_1e_2$ or through its inverse $e_2^{-1}e_1^{-1}$. They correspond therefore to words $\tilde{w}_f$ and $\tilde{w}'_i$ in $e_{12},e_3,\ldots,e_k$ which are the edges of $\Gbb'$. By~\eqref{eq:YM_distrib}, we have on the one hand
\begin{align*}
\int_{\Omega^1(\Gbb,G)}F(\mathcal{R}(U))d\mu_{\Gbb,G,Q,\vol,\mathfrak{t}}(U)= & \int_{G^k}F(U_{e_1}U_{e_2},U_{e_3},\ldots,U_{e_k})\prod_{f\in F}Q_{\vert f\vert}(w_f(U_{e_1},\ldots,U_{e_k}))\\
&\times\prod_{i=1}^n\delta_{[t_i]}(w'_i(U_{e_1},\ldots,U_{e_k}))dU_{e_1}\ldots dU_{e_k}.
\end{align*}
We perform a change of variables $U_{e_{12}}=U_{e_1}U_{e_2}$, which replaces all words $w_f$ and $w'_i$ by $\tilde{w}_f$ and $\tilde{w}'_i$, then integrate over $U_{e_2}$, yielding $\int_{\Omega^1(\Gbb',G)}F(U)d\mu_{\Gbb',G,Q,\vol,\mathfrak{t}}(U)$.
\end{proof}

We shall use the following convention for convolution. If $P,Q\in L^1(G)$, then
\[
 (P*Q)(g)=\int_G P(x)Q(x^{-1}g)\,dx .
\]

\begin{proposition}[Edge removal]\label{prop:edge_removal}
Let $\mathbb G=(V,E,F)$ be a topological map embedded in $\Sigma_{g,n}$.
Let $e$ be an internal edge incident to two distinct faces $f_1,f_2$, and
let $\mathbb G'$ be obtained by deleting $e$ and merging these faces into
$f_{12}$. Let $\mathcal R$ forget the variable attached to $e$. Then, for
every bounded measurable test function $\Phi$ on
$\Omega^1(\mathbb G',G)$,
\begin{align*}
 &\int \Phi(\mathcal R(U))\,
 d\mu_{\mathbb G,G,Q,\vol,\mathfrak t}(U)\\
 &\quad=
 \int \Phi(U')
 (Q_{|f_1|}*Q_{|f_2|})
 \bigl(\Omega^{U'}(f_{12})\bigr)
 \prod_{f\in F'\setminus\{f_{12}\}}
 Q_{|f|}\bigl(\Omega^{U'}(f)\bigr)
 \prod_i\delta_{[t_i]}(\Hol_{\partial b_i}(U'))\,dU'.
\end{align*}
In particular, if $(Q_t)_{t>0}$ is a convolution semigroup, then
\[
 \mathcal R_*\mu_{\mathbb G,G,Q,\vol,\mathfrak t}
 =\mu_{\mathbb G',G,Q,\vol,\mathfrak t}.
\]
If, in addition, $Q_{|f_1|+|f_2|}>0$ almost everywhere, the first identity
may equivalently be written as a Radon--Nikodym formula with quotient
$(Q_{|f_1|}*Q_{|f_2|})/Q_{|f_1|+|f_2|}$.
\end{proposition}

\begin{proof}
Orient $e$ so that
\[
 \partial f_1=\alpha e,
 \qquad
 \partial f_2=e^{-1}\beta,
\]
up to cyclic permutation, where $\alpha,\beta$ do not involve $e$. If
$x=U_e$, $a=\Hol_\alpha(U)$ and $b=\Hol_\beta(U)$, the only factors that
depend on $x$ are
\[
 Q_{|f_1|}(ax)Q_{|f_2|}(x^{-1}b).
\]
All boundary-holonomy constraints and all other face factors depend only on
$U'=\mathcal R(U)$. Haar invariance and the change of variables $y=ax$
give
\[
 \int_GQ_{|f_1|}(ax)Q_{|f_2|}(x^{-1}b)\,dx
 =
 (Q_{|f_1|}*Q_{|f_2|})(ab),
\]
and $ab=\Omega^{U'}(f_{12})$. This proves the displayed push-forward
identity. The semigroup case follows from
$Q_{|f_1|}*Q_{|f_2|}=Q_{|f_1|+|f_2|}$.
\end{proof}

\section{Proof of state-sum formulas}

In this section, we use the following consequence of the Peter--Weyl
theorem. The irreducible characters $(\chi_\lambda)_{\lambda\in\widehat G}$
form an orthonormal Hilbert basis of
$L^2(G)^{\operatorname{Ad}G}$. Whenever a central function $Q$ is
Peter--Weyl regular, namely
\[
\sum_{\lambda\in\widehat G} d_\lambda^2 |a_Q(\lambda)|<\infty,
\qquad
a_Q(\lambda)=\frac1{d_\lambda}\langle Q,\chi_\lambda\rangle,
\]
its character expansion
\[
Q(x)=\sum_{\lambda\in\widehat G}d_\lambda a_Q(\lambda)\chi_\lambda(x)
\]
converges absolutely and uniformly. This is only a convenient
one-plaquette sufficient condition. The state-sum formulas below require
instead the weaker graph-level product summability conditions
\eqref{eq:PF-spectral-summability} and
\eqref{eq:component-spectral-summability}.

\subsection{The partition function}

We use three elementary reductions. The first is Schur orthogonality in
character form:
\begin{equation}\label{eq:character-face-gluing}
 \int_G\chi_\lambda(ax)\chi_\mu(x^{-1}b)\,dx
 =
 \delta_{\lambda\mu}\frac{\chi_\lambda(ab)}{d_\lambda}.
\end{equation}

\begin{lemma}[Tree gauge fixing]\label{lem:tree-gauge-fixing}
Let $\Gamma=(V,E)$ be a finite connected graph, let $T\subset E$ be a
spanning tree, and fix a root $o\in V$. If $\Phi:G^E\to\mathbb C$ is
integrable and invariant under the gauge action of $G^V$, then
\[
 \int_{G^E}\Phi(U)\,dU
 =
 \int_{G^{E\setminus T}}\Phi(U^T)\,dU,
\]
where $U^T_e=1_G$ for $e\in T$ and $U^T_e=U_e$ otherwise.
The same statement holds for the boundary-conditioned integrals of
Definition~\ref{def:discrete-YM-measure}.
\end{lemma}

\begin{proof}
Orient every edge of $T$ away from $o$. For a connection $U$, there is a
unique gauge transformation $j_U$ with $j_U(o)=1_G$ such that
$(j_U\cdot U)_e=1_G$ for all $e\in T$; its values are obtained recursively
along the tree. The change of variables
\[
 U\longmapsto\bigl((j_U\cdot U)|_{E\setminus T},
 (j_U(v))_{v\ne o}\bigr)
\]
is triangular and preserves product Haar measure. Gauge invariance removes
the second set of variables, whose Haar volume is one. The boundary
constraints are invariant under gauge transformations, so the same argument
applies to the iterated integral~\eqref{eq:DS-rigorous}.
\end{proof}

\begin{lemma}[One-face normal form]\label{lem:one-face-normal-form}
Let $\mathbb G_0$ be a connected one-face topological map on the compact
oriented surface $\Sigma_{g,n}$ and suppose that every boundary component is
carried by $\mathbb G_0$. After contracting a spanning tree, the resulting
one-vertex ribbon graph has $2g+n$ oriented loop edges. There is a free basis
\[
 a_1,b_1,\ldots,a_g,b_g,c_1,\ldots,c_n
\]
of its edge-path group such that, with the face and boundary orientations induced by the orientation of the surface, the unique face word is, up to cyclic
permutation,
\[
 \prod_{i=1}^g[a_i,b_i]\prod_{j=1}^n c_j,
\]
and $c_j$ represents the $j$-th oriented boundary component. The
corresponding change of edge variables preserves product Haar measure.
\end{lemma}

\begin{proof}
The thickening of the one-vertex ribbon graph is $\Sigma_{g,n}$ with the
interior of the unique face removed. The standard polygonal classification
of compact oriented surfaces supplies the displayed system of oriented
generators and boundary word. Any two free bases of the one-vertex graph are
related by elementary Nielsen transformations. On $G^{2g+n}$ these are
generated by permutations of variables, inversion of a variable, and
replacements $x_i\mapsto x_ix_j$ with $i\ne j$; each preserves product Haar
measure by translation and inversion invariance. Hence the normal-form
change of variables is measure preserving.
\end{proof}

\begin{proof}[Proof of Theorem~\ref{thm:Fourier_PF}]
Write $F=\{f_1,\ldots,f_q\}$. For $\varepsilon>0$, regularize each
plaquette density by the heat kernel:
\[
 Q_{|f|}^{(\varepsilon)}
 :=Q_{|f|}*p_\varepsilon.
\]
These are central smooth probability densities, hence Peter--Weyl
regular, and their normalized Fourier coefficients are
\[
 a_{|f|}^{(\varepsilon)}(\lambda)
 =e^{-\frac{\varepsilon}{2}c_2(\lambda)}
 a_{|f|}(\lambda).
\]
Let $Z_\varepsilon$ denote the partition function obtained by replacing
every face weight by $Q_{|f|}^{(\varepsilon)}$. We may multiply the
absolutely and uniformly convergent character expansions and integrate
term by term.

Consider the dual adjacency graph whose vertices are the faces of
$\mathbb G$ and whose edges correspond to internal primal edges separating
two distinct faces. This dual graph is connected: a
path in the interior of the connected surface joining points in two
faces, chosen transverse to the graph and avoiding its vertices, gives a
chain of adjacent faces. Choose a spanning tree $T^*$ of this dual graph.
Integrate successively the $q-1$ primal edge variables dual to $T^*$. An example is displayed in Figure~\ref{fig:tree-integration}. At
each step, Schur orthogonality in the form
\eqref{eq:character-face-gluing} forces the two incident face labels to
agree, merges the character words, and contributes one factor
$d_\lambda^{-1}$. Since $T^*$ connects all faces, all labels are equal to
a single $\lambda$. The $q$ factors $d_\lambda$ from the character
expansions and the $q-1$ gluing factors leave one factor $d_\lambda$. We
obtain
\begin{equation}\label{eq:PF-one-face-intermediate}
 Z_\varepsilon(t_1,\ldots,t_n)
 =
 \sum_{\lambda\in\widehat G}
 e^{-\frac{q\varepsilon}{2}c_2(\lambda)}
 \left(\prod_{f\in F}a_{|f|}(\lambda)\right)d_\lambda I_\lambda,
\end{equation}
where $I_\lambda$ is the boundary-conditioned integral of
$\chi_\lambda$ of the boundary word of the unique remaining face.

\begin{figure}[h!]
 \centering
 \includegraphics[width=0.9\linewidth]{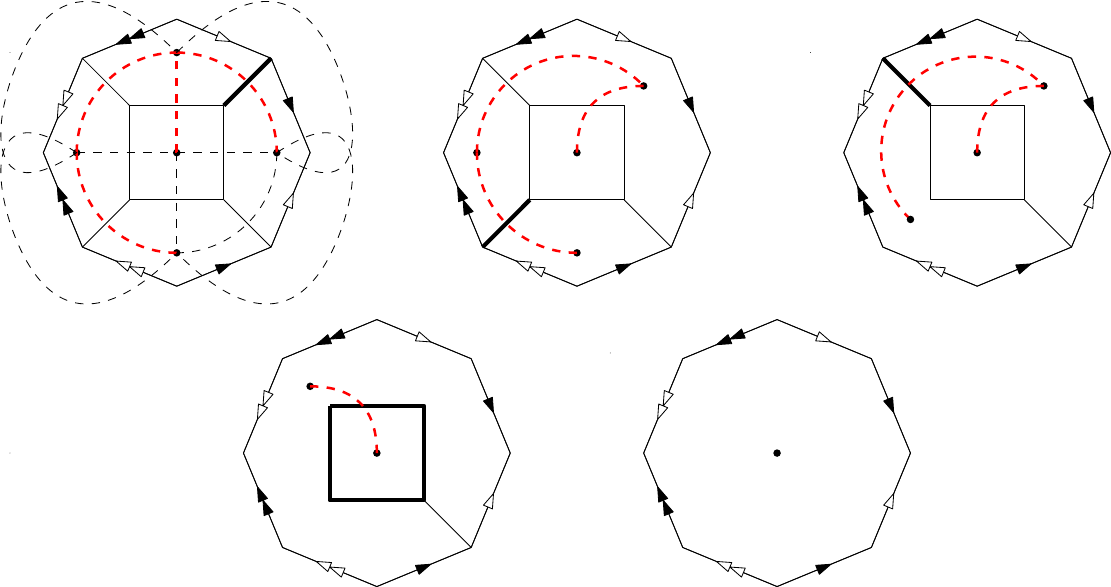}
 \caption{An example of the successive integration for a topological map in $\Sigma_{2,0}$. The dual graph is dashed, and the dual tree is in red. The oriented edges are glued pairwise. In the first step, the whole dual graph is displayed, but all non-tree dual edges are omitted in the next steps. Each step consists in integrating the corresponding primal edge (in bold) to an edge of the dual tree, thereby merging two faces of the primal graph.}
 \label{fig:tree-integration}
\end{figure}

The primal graph remaining after these deletions is connected and has one
face. Apply Lemma~\ref{lem:tree-gauge-fixing} to a spanning tree and then
Lemma~\ref{lem:one-face-normal-form}. Thus
\[
 I_\lambda
 =
 \int_{G^{2g+n}}
 \chi_\lambda\!\left(
 \prod_{i=1}^g[A_i,B_i]\prod_{j=1}^n C_j
 \right)
 \prod_{j=1}^n\delta_{[t_j]}(dC_j)
 \prod_{i=1}^g dA_i\,dB_i.
\]
Using orbital coordinates $C_j=U_jt_jU_j^{-1}$ and Schur orthogonality,
\[
 \int_{G^2}\chi_\lambda([A,B]x)\,dA\,dB
 =\frac{\chi_\lambda(x)}{d_\lambda^2},
 \qquad
 \int_G\chi_\lambda(xUtU^{-1})\,dU
 =\frac{\chi_\lambda(x)\chi_\lambda(t)}{d_\lambda}.
\]
Iterating these identities yields
\[
 I_\lambda
 =d_\lambda^{-2g+1-n}
 \prod_{j=1}^n\chi_\lambda(t_j).
\]
Substitution in~\eqref{eq:PF-one-face-intermediate} gives
\begin{align*}
 Z_\varepsilon(t_1,\ldots,t_n)
 &=
 \sum_{\lambda\in\widehat G}
 e^{-\frac{q\varepsilon}{2}c_2(\lambda)}
 \left(\prod_{f\in F}a_{|f|}(\lambda)\right)
 d_\lambda^{2-2g}
 \prod_{j=1}^n\frac{\chi_\lambda(t_j)}{d_\lambda}.
\end{align*}

Since $Q_{|f|}*p_\varepsilon\to Q_{|f|}$ uniformly on $G$ for each
face, the finite product of plaquette densities converges uniformly and
$Z_\varepsilon\to Z_{\mathbb G,G,Q,\vol}$. On the series side,
$|\chi_\lambda(t_j)|\leq d_\lambda$ and
$e^{-q\varepsilon c_2(\lambda)/2}\leq1$, so the absolute value of the
$\lambda$-term is bounded by
\[
 d_\lambda^{2-2g}
 \prod_{f\in F}|a_{|f|}(\lambda)|,
\]
which is summable by~\eqref{eq:PF-spectral-summability}. Dominated
convergence as $\varepsilon\downarrow0$ proves
\eqref{eq:pf-state-sum-intro} and its absolute convergence.
\end{proof}

\subsection{Wilson loop integrals}

Fix a vertex $o$ of $\mathbb G$ and let $L=(\ell_1,\ldots,\ell_k)$ be a finite regular family of graph loops in $L_o(\mathbb G)$. Let $\Gamma(L)\subset \mathbb G$ be the subgraph formed by the edges visited by the loops. We assume that cutting $\Sigma_{g,n}$ along $\Gamma(L)$ gives finitely many compact surfaces with boundary $C_1,\ldots,C_r .$ An illustration is given in Figure~\ref{fig:ribbon-cut} below (note that only the ribbon graph associated to $L$ is displayed, not the whole topological map $\Gbb$).

\begin{figure}[h!]
 \centering
 \includegraphics[width=\linewidth, trim={-2cm 0cm 0cm 0cm}]{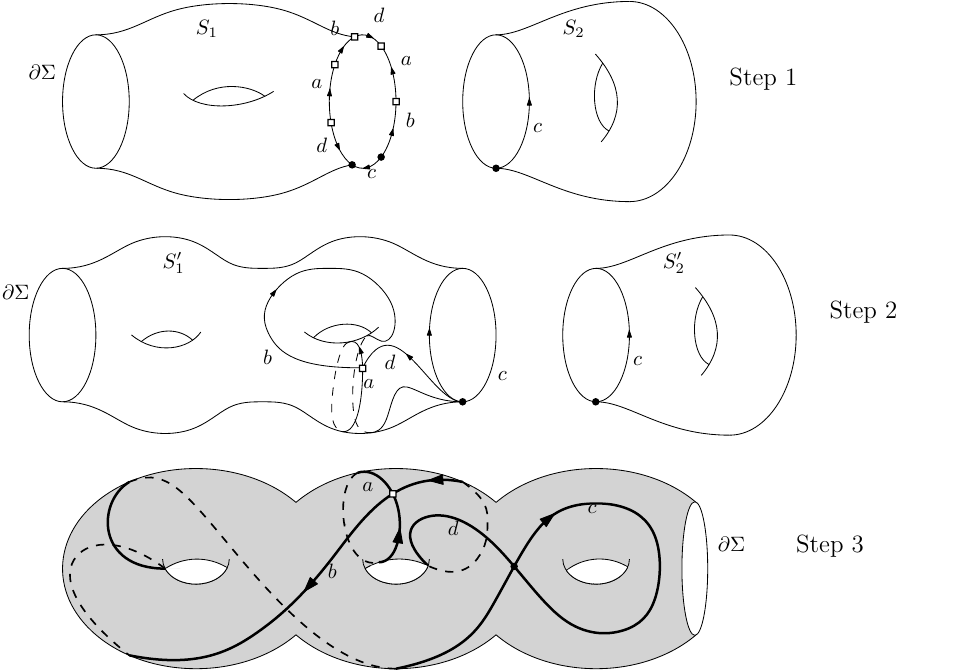}
 \caption{The surface decomposition of the ribbon graph represented by the loop configuration from Figure~\ref{fig:ribbon}. Start with two surfaces $S_1$ and $S_2$ of genus $1$ with boundaries $C_1\cup\partial\Sigma$ and $C_2$, with $C_1=abda^{-1}b^{-1}cd^{-1}$, $C_2=c^{-1}$ (step 1). Glue copies of edges according to their labels and orientation, which first produces two surfaces $S_1'$ and $S_2'$ (step 2), and finally the whole surface $\Sigma$, with a ribbon graph that carries two vertices and four edges. Vertices are coloured and edges are labelled in order to follow the steps.}
 \label{fig:ribbon-cut}
\end{figure}

Let $\mathbb G_\alpha$ be the graph induced by $\mathbb G$ on $C_\alpha$, and let $F_\alpha$ be the set of faces of $\mathbb G$ contained in $C_\alpha$. Choose an orientation of every edge of $\Gamma(L)$. If $U=(U_e)_{e\in E_L}\in G^{E_L}$ is a configuration on the loop edges, every boundary component of every $C_\alpha$ has a holonomy
\[
 x_{\alpha,q}(U)\in G,\qquad 1\leq q\leq m_\alpha .
\]
These boundary holonomies are of two types:
\begin{enumerate}
\item if the boundary component comes from an original boundary circle of
$\Sigma_{g,n}$, then it has the form $x_{\alpha,q}(U)=t_j$ up to conjugacy,
\item if it comes from a side of the cut along $L$, then it is a word in the variables $U_e^{\pm1}$.
\end{enumerate}
Likewise, for each loop $\ell_i$, we denote by $h_{\ell_i}(U)$ the corresponding word in the variables $U_e^{\pm1}$.

Recall from the introduction that, for a labelling $\Lambda:\{C_1,\ldots,C_r\}\longrightarrow \widehat G,$ we define the spectral coefficient
\[
 \kappa_{\mathbb G,Q,\mathrm{vol}}(\Lambda)
 =
 \prod_{\alpha=1}^r
 \prod_{f\in F_\alpha}
 a_{|f|}\bigl(\Lambda(C_\alpha)\bigr),
\]
where $a_t(\lambda)=\frac{1}{d_\lambda}\langle Q_t,\chi_\lambda\rangle$, and the topological coefficient by
\[
\begin{aligned}
\widehat W_{F,L}(\Lambda)=\int_{G^{E_L}}F\bigl(h_{\ell_1}(U),\ldots,h_{\ell_k}(U)\bigr)\prod_{\alpha=1}^r\left[d_{\Lambda(C_\alpha)}^{\chi(C_\alpha)}\prod_{q=1}^{m_\alpha}\chi_{\Lambda(C_\alpha)}\bigl(x_{\alpha,q}(U)\bigr)\right]\,dU.
\end{aligned}
\]
Here $dU$ is the product Haar probability measure on $G^{E_L}$.

\begin{proof}[Proof of Theorem~\ref{thm:Fourier_Wilsonloop}]
Let $E_L\subset E$ be the set of edges belonging to the loop graph $\Gamma(L)$. For each $\alpha$, let $E_\alpha'$ be the set of edges of the induced graph on $C_\alpha$ which do not belong to the cut graph; this includes edges carried by an original boundary component of $\Sigma_{g,n}$. Then
\[
 E=E_L\sqcup E_1'\sqcup\cdots\sqcup E_r'.
\] The face set decomposes as
\[
 F=F_1\sqcup\cdots\sqcup F_r .
\]
For a fixed cut-edge configuration
\[
 U=(U_e)_{e\in E_L}\in G^{E_L},
\]
the integral over the remaining variables factorizes over the components $C_\alpha$. Indeed, once the variables $U$ are fixed, each component $C_\alpha$ is an ordinary lattice gauge theory on $\mathbb G_\alpha$, with plaquette weights $Q_{|f|}$, $f\in F_\alpha$, and with boundary holonomies
\[
 x_{\alpha,1}(U),\ldots,x_{\alpha,m_\alpha}(U).
\]
Therefore
\[
\begin{aligned}
&\int_{\Omega^1(\mathbb G,G)}
 F\bigl(\operatorname{Hol}_{\ell_1}(U),\ldots,
 \operatorname{Hol}_{\ell_k}(U)\bigr)\,
 d\mu_{\mathbb G,G,Q,\mathrm{vol},\mathfrak t}(U) \\
&\qquad =
\int_{G^{E_L}}
 F\bigl(h_{\ell_1}(U),\ldots,h_{\ell_k}(U)\bigr)
 \prod_{\alpha=1}^r
 Z_{\mathbb G_\alpha,G,Q,\mathrm{vol}}
 \bigl(x_{\alpha,1}(U),\ldots,x_{\alpha,m_\alpha}(U)\bigr)
 \,dU .
\end{aligned}
\]
Here the notation for $Z_{\mathbb G_\alpha,G,Q,\mathrm{vol}}$ means the partition function of the component $C_\alpha$ with the displayed boundary holonomies. For boundary components inherited from $\partial\Sigma_{g,n}$, the holonomy is fixed in the conjugacy class $[t_j]$; for cut boundary components, it is the corresponding word in the variables $U_e^{\pm1}$.

By the partition-function formula applied to $C_\alpha$, we have
\[
\begin{aligned}
Z_{\mathbb G_\alpha,G,Q,\mathrm{vol}}
 \bigl(x_{\alpha,1},\ldots,x_{\alpha,m_\alpha}\bigr)
=
\sum_{\lambda_\alpha\in\widehat G}
 \left(
 \prod_{f\in F_\alpha}
 a_{|f|}(\lambda_\alpha)
 \right)
 d_{\lambda_\alpha}^{\chi(C_\alpha)}
 \prod_{q=1}^{m_\alpha}
 \chi_{\lambda_\alpha}(x_{\alpha,q}).
\end{aligned}
\]
Indeed this is the formula of Theorem~\ref{thm:Fourier_PF}. Substituting this expansion for every component $C_\alpha$, and collecting the labels $\Lambda(C_\alpha)=\lambda_\alpha,$ we obtain
\[
\begin{aligned}
&\int_{\Omega^1(\mathbb G,G)}
 F\bigl(
 \operatorname{Hol}_{\ell_1}(U),\ldots,
 \operatorname{Hol}_{\ell_k}(U)
 \bigr)\,
 d\mu_{\mathbb G,G,Q,\mathrm{vol},\mathfrak t}(U) \\
&\quad =
\sum_{\Lambda:\{C_1,\ldots,C_r\}\to\widehat G}
 \left[
 \prod_{\alpha=1}^r
 \prod_{f\in F_\alpha}
 a_{|f|}\bigl(\Lambda(C_\alpha)\bigr)
 \right]   \\
&\qquad\qquad \times
 \int_{G^{E_L}}
 F\bigl(h_{\ell_1}(U),\ldots,h_{\ell_k}(U)\bigr)
 \prod_{\alpha=1}^r
 \left[
 d_{\Lambda(C_\alpha)}^{\chi(C_\alpha)}
 \prod_{q=1}^{m_\alpha}
 \chi_{\Lambda(C_\alpha)}
 \bigl(x_{\alpha,q}(U)\bigr)
 \right]
 dU .
\end{aligned}
\]
The component expansions converge absolutely and uniformly in their
boundary variables. Indeed, if $C_\alpha$ has genus $g_\alpha$ and
$m_\alpha$ boundary components, then the absolute value of its
$\lambda$-term is bounded uniformly in the boundary holonomies by
\[
 d_\lambda^{\chi(C_\alpha)+m_\alpha}
 \prod_{f\in F_\alpha}|a_{|f|}(\lambda)|
 =d_\lambda^{2-2g_\alpha}
 \prod_{f\in F_\alpha}|a_{|f|}(\lambda)|,
\]
which is summable by~\eqref{eq:component-spectral-summability}. Since
there are finitely many components and $F$ is bounded, Tonelli's
theorem justifies the exchange of all label sums with the cut-edge
integral. This is exactly~\eqref{eq:decomp_spectrale_WL}.
\end{proof}

\begin{remark}\label{rem:topological-coefficient-subdivision}
The topological coefficient is unchanged if an edge of the loop graph is subdivided. Indeed, if an oriented edge $e$ is replaced by two consecutive edges $e_1e_2$, then every loop word and every cut-boundary word depends on the corresponding variables only through the product $U_{e_1}U_{e_2}$. The push-forward of product Haar measure under
\[
 G^2\longrightarrow G,
 \qquad
 (U_{e_1},U_{e_2})\longmapsto U_{e_1}U_{e_2},
\]
is Haar measure. Hence integration over one of the two variables reduces the subdivided coefficient to the original coefficient, exactly as in the proof of Proposition~\ref{prop:edge_fusion}. By iteration, $\widehat W_{F,L}(\Lambda)$ depends only on the minimal marked loop graph, with unmarked vertices of degree two suppressed.
\end{remark}

\section{Holonomy processes and scaling limits}
\label{sec:loop-space-regular-processes}

In this section, we prove the universality theorem, Theorem~\ref{thm:universality}. Throughout, $\Sigma$ is a compact connected oriented Riemannian surface, possibly with smooth boundary, endowed with a smooth positive area density. We fix a base point $o\in\operatorname{int}(\Sigma).$ The structure group $G$ is a compact connected Lie group endowed with a bi-invariant distance $d_G$. Unless explicitly stated otherwise, every regular loop family is based at $o$ and has image contained in $\operatorname{int}(\Sigma)$. This restriction concerns only the regular test families used to formulate convergence; arbitrary rectifiable based loops are recovered by approximation and stochastic continuity.

\subsection{Regular loop families and stable approximations}

We denote by $P(\Sigma)$ the set of rectifiable paths in $\Sigma$, modulo the equivalence relation used by L\'evy, and by $P_{x,y}(\Sigma)$ the set of such paths from $x$ to $y$. The concatenation of two (equivalent classes of) paths $\gamma_1$ and $\gamma_2$ is denoted by $\gamma_1\gamma_2$, and the inverse of a path $\gamma$ is denoted by $\gamma^{-1}$. The set of rectifiable loops based at $o$ is $L_o(\Sigma):=P_{o,o}(\Sigma).$ We use L\'evy's metric $d_1$ on rectifiable paths, as defined in \cite[Section~1.2.2]{Lev10} that we recall here.

For $\gamma_1,\gamma_2\in P(\Sigma)$, we define the distance
\begin{equation}
 d_1(\gamma_1,\gamma_2)
 = \inf_{\text{param}} \left\{
 \sup_{t\in [0,1]} d\big(\gamma_1(t),\gamma_2(t)\big)
 + \int_0^1 d_{T\Sigma}\big(\dot{\gamma}_1(t),\dot{\gamma}_2(t)\big)\,dt
 \right\},
\end{equation}
where the infimum is taken over all reparametrizations of $\gamma_1$ and $\gamma_2$ on $[0,1]$. 
The quantity $d_{T\Sigma}(X,Y)$, for $X\in T_x\Sigma$ and $Y\in T_y\Sigma$ with $x,y\in \Sigma$, is defined by
\begin{equation}
 d_{T\Sigma}(X,Y)^2= 
 \begin{cases}
 d(x,y)^2 + \big\|P_{x\to y}X - Y\big\|_{T_y\Sigma}^2 
 & \text{if $x$ and $y$ can be joined by a unique geodesic}, \\
 1 & \text{otherwise}.
 \end{cases}
\end{equation}
Here $P_{x\to y}$ denotes the parallel transport operator along the unique geodesic joining $x$ and $y$. 
For a precise discussion of why this defines a well-posed metric, see \cite[Section~1.2.2]{Lev10}.
\begin{theorem}[\cite{Lev10}, Propositions~1.2.12 and~1.2.14]
\label{thm:Levy-d1-complete-rewrite}
For every $x,y\in\Sigma$, the metric space $(P_{x,y}(\Sigma),d_1)$ is complete. Moreover, piecewise geodesic paths with fixed endpoints are dense in $P_{x,y}(\Sigma)$ for $d_1$.
\end{theorem}

\begin{proof}
L\'evy's propositions give completeness of $P(\Sigma)$ and density of piecewise geodesic paths in the corresponding fixed-endpoint classes. The initial- and terminal-point maps are continuous for $d_1$, directly from the definition of that metric. Hence $P_{x,y}(\Sigma)$ is closed in $P(\Sigma)$ and is therefore complete.
\end{proof}
\begin{definition}\label{def:holonomy-process}
A holonomy process on the pointed surface $(\Sigma,o)$ with structure group $G$ is a family $H=(H_\ell)_{\ell\in L_o(\Sigma)}$ of $G$-valued random variables on a common probability space such that, almost surely,
\[ \begin{cases}
 H_{1_o}=1_G \\
 \forall \ell_1,\ell_2\in L_o(\Sigma), H_{\ell_1\ell_2}=H_{\ell_1}H_{\ell_2} \\
 \forall \ell\in L_o(\Sigma), H_{\ell^{-1}}=H_\ell^{-1}
\end{cases}.
\]
Equivalently, it is a random element of $\Hom(L_o(\Sigma),G)$ equipped with the cylinder $\sigma$-field generated by the evaluation maps. If $\Gbb$ is a topological map in $\Sigma$ containing $o$ as a vertex, a holonomy process on $\Gbb$ is defined in the same way with $L_o(\Gbb)$ in place of $L_o(\Sigma)$.
\end{definition}

Our terminology is deliberately different from the one introduced in \cite{Lev10}. A holonomy process is a single multiplicative stochastic process on the based-loop group of one fixed pointed surface (or one fixed graph), whereas a two-dimensional Markovian holonomy field consists of a coherent family of finite measures, one for each measured marked surface with $G$-constraints, on multiplicative $G$-valued functions of paths. These measures satisfy, among other properties, covariance under the relevant surface maps, compatibility with changes of constraints, factorization under disjoint unions, and cutting--gluing axioms encoding the two-dimensional Markov property. Regular Markovian holonomy fields satisfy in addition the continuity requirements of \cite[Definition~3.1.3]{Lev10}.

Fixing a measured surface $\Sigma$ and choosing a base point $o$, the corresponding Markovian holonomy field canonically determines a finite gauge-invariant measure on based-loop holonomies, hence a measure on $\Hom(L_o(\Sigma),G)/G$. Whenever its total mass is non-zero, we normalize it. Choosing a conjugation-invariant realization of this quotient law then gives a holonomy process in the sense of Definition~\ref{def:holonomy-process}. The choice of realization is immaterial for the gauge-invariant observables considered in this paper. Conversely, multiplicativity on one fixed surface alone does not imply the covariance, consistency, or Markov axioms required of a Markovian holonomy field. Thus a Markovian holonomy field determines fixed-surface holonomy processes up to the natural gauge equivalence, but the two notions are not synonymous.

Whenever $\psi$ is the symbol of an admissible L\'evy process, L\'evy's construction \cite[Theorem~4.3.1]{Lev10} gives a regular Markovian holonomy field. Throughout the rest of this paper, $H^\psi$ denotes the holonomy process on the fixed pointed surface $(\Sigma,o)$ induced by that field (with the prescribed boundary constraints, when present).

The index sets of a graph process and of a surface process are different. Rather than extending graph holonomies noncanonically to all surface loops, we approximate finite loop configurations while preserving their embedded topology.

\begin{definition}\label{def:regular-loop-family}
A finite family of based rectifiable loops $L=(\ell_1,\ldots,\ell_r)$ is called \emph{regular} if its image is contained in $\operatorname{int}(\Sigma)$ and the union of the images of its loops is a finite embedded ribbon graph $\Gamma(L)\subset\operatorname{int}(\Sigma).$ The vertices are the base point, the self-intersection points and the mutual intersection points. The edges are embedded piecewise-smooth arcs with pairwise disjoint interiors. We always take $\Gamma(L)$ to be minimal: all unmarked vertices of degree two are suppressed.
\end{definition}

Note that the ribbon graph drawn in Figure~\ref{fig:ribbon} is in fact $\Gamma(L)$ for a regular family $L=(\ell_1,\ell_2)$ of loops in $\Sigma_{3,1}$.

\begin{definition}\label{def:same-ribbon-type}
Two finite regular based loop families
\[
 L=(\ell_1,\ldots,\ell_r),
 \qquad
 L'=(\ell'_1,\ldots,\ell'_r),
\]
have the same \emph{marked ambient ribbon type} if there exists an orientation-preserving homeomorphism $h:\Sigma\longrightarrow\Sigma$ isotopic to the identity relative to $\{o\}\cup\partial\Sigma$\footnote{Meaning that there exists an isotopy from the identity to $h$ that fixes $o$ and $\partial\Sigma$ pointwise at every time.}, such that
\[
 h(\Gamma(L))=\Gamma(L')
\]
and, for every $i$, $h\circ\ell_i=\ell'_i$ as elements of $P(\Sigma)$. Such a homeomorphism induces a canonical bijection between the connected components of the complements, as well as between the corresponding cut-open surfaces and their marked boundary cycles.
\end{definition}

\begin{definition}\label{def:mesh-rewrite}
Let $\Gbb=(V,E,F)$ be a topological map embedded in $\Sigma$. We write
\[
 \mesh(\Gbb)
 =
 \max\left\{
 \max_{e\in E}\ell(e),
 \max_{f\in F}\operatorname{diam}(\overline f)
 \right\}.
\]
A sequence $(\Gbb_n)_{n\geq1}$ has mesh tending to zero if $\mesh(\Gbb_n)\to0$. We always require $o\in V(\Gbb_n)$.
\end{definition}

\begin{lemma}\label{lem:mesh-implies-face-area}
If $\mesh(\Gbb_n)\to0$, then
\[
 \max_{f\in F(\Gbb_n)}|f|\xrightarrow[n\to\infty]{}0.
\]
\end{lemma}

\begin{proof}
Every face of diameter at most $\frac r2$ is contained in a geodesic ball of radius $r$. Since the area density is smooth on the compact surface,
\[
 \sup_{x\in\Sigma}\operatorname{vol}(B(x,r))\longrightarrow0
 \qquad (r\downarrow0).
\]
The assertion follows from the definition of the mesh.
\end{proof}

\begin{definition}\label{def:Levy-stable-graph-approximation}
Let $L=(\ell_1,\ldots,\ell_r)$ be a regular based loop family. A \emph{stable graph approximation} of $L$ is the data of a sequence of embedded topological maps $(\Gbb_n)_{n\geq1}$, and a sequence of graph-loop families
\[
 L_n=(\ell_{1,n},\ldots,\ell_{r,n}),
 \qquad
 \ell_{i,n}\in L_o(\Gbb_n),
\]
such that:
\begin{enumerate}
\item $\mesh(\Gbb_n)\xrightarrow[n\to\infty]{}0$;
\item for every $i=1,\dots,r$,
\[
 d_1(\ell_{i,n},\ell_i)\xrightarrow[n\to\infty]{}0;
\]
\item for all sufficiently large $n$, there exists a homeomorphism $h_n$ such that $L_n$ and $L$ have the same marked ambient ribbon type; and if $C_1,\ldots,C_m$ are the connected components of $\Sigma\setminus\Gamma(L)$ and
\[
 C_{j,n}:=h_n(C_j),
\]
then
\[
 |C_{j,n}|\longrightarrow|C_j|,
 \qquad j=1,\ldots,m.
\]
\end{enumerate}
We say that $L$ has a stable graph approximation carried by the sequence $(\mathbb{G}_n)$.
The same notation is used for the corresponding cut-open compact surfaces with boundary.
\end{definition}

An illustration is given in Figure~\ref{fig:loop-to-approx}.

\begin{figure}[!h]
 \centering
 \includegraphics[width=0.9\linewidth]{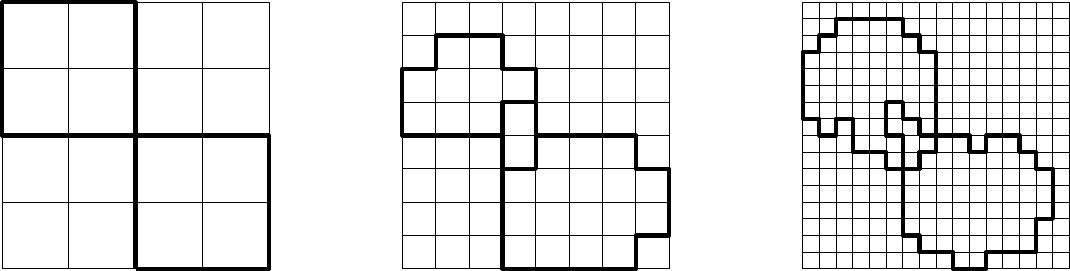}
 \caption{A stable graph approximation, using square lattices, of a configuration of two circles with a small overlap. The overlap is not visible in the first graph but appears in the second and third ones.}
 \label{fig:loop-to-approx}
\end{figure}

\begin{lemma}\label{lem:regular-density-d1}
Let $\ell_1,\ldots,\ell_r\in L_o(\Sigma)$ be rectifiable based loops. Then there exist regular based loop families $L^{(k)}=(\ell_1^{(k)},\ldots,\ell_r^{(k)})$ with image contained in $\operatorname{int}(\Sigma)$ such that, for every $i=1,\dots,r$
\[
 d_1(\ell_i^{(k)},\ell_i)\longrightarrow0.
\]
\end{lemma}

\begin{proof}
By Theorem~\ref{thm:Levy-d1-complete-rewrite}, first approximate the loops simultaneously, with the base point fixed, by piecewise geodesic loops. Choose a smooth collar
\[
 c:\partial\Sigma\times[0,\rho)\longrightarrow\Sigma
\]
whose image does not contain $o$. For sufficiently small $\delta>0$, choose a smooth function $\theta_\delta:[0,\rho)\to(0,\rho)$ such that $\theta_\delta(s)=s$ for $s\ge\rho/2$ and $\theta_\delta\to\operatorname{id}$ in $C^1$ as $\delta\downarrow0$. Define $p_\delta:\Sigma\to\Sigma$ by
\[
 p_\delta(c(x,s))=c(x,\theta_\delta(s))
\]
in the collar and by $p_\delta=\operatorname{id}$ outside it. Then $p_\delta(o)=o$, the image of every path $p_\delta\circ\ell_i$ is contained in $\operatorname{int}(\Sigma)$ since $\theta_\delta$ is on $(0,\rho)$, and $p_\delta\circ\ell_i\to\ell_i$ in $d_1$ as $\delta\downarrow0$.
Subdivide the resulting loops into finitely many smooth embedded arcs. Keeping their endpoints fixed, perturb their interiors successively and arbitrarily slightly in the $C^1$ topology so that every pair of unrelated arcs is transverse, every arc avoids all unrelated vertices, no three arcs meet away from $o$, and no two arcs share a nontrivial interval. In a small disk around $o$, keep $o$ fixed and choose pairwise disjoint initial and terminal germs with the prescribed cyclic order, so that no intersections accumulate at the base point. These conditions are open after each finite step, hence all perturbations can be chosen simultaneously small. The intersection set away from $o$ is then a compact discrete set and is therefore finite. The resulting family is regular. Since edgewise $C^1$ convergence implies convergence in $d_1$, the perturbations may be made arbitrarily small in $d_1$.
\end{proof}

\subsection{Universally stable graph sequences}
\label{subsec:universal-stable-graph-sequences}

\begin{definition}\label{def:universally-stable-graph-sequence}
A sequence of embedded topological maps $(\Gbb_n)_{n\geq1}$ is called \emph{universally stable for regular loop families} if:
\begin{enumerate}
\item $o\in V(\Gbb_n)$ for every $n$, every boundary component of $\Sigma$ is carried by $\Gbb_n$, and $\mesh(\Gbb_n)\to0$;
\item every finite regular based loop family admits a stable graph approximation carried by the sequence $(\Gbb_n)$.
\end{enumerate}
\end{definition}

The main result of this subsection is the following proposition, which asserts the existence of a universally stable graph sequence.

\begin{proposition}\label{prop:existence-universally-stable-sequence}
Let $\Sigma$ be a compact smooth oriented Riemannian surface, possibly with smooth boundary, and let $o\in\operatorname{int}(\Sigma)$. There exists a sequence $(\mathcal T_n)_{n\geq1}$ of smooth triangulations of $\Sigma$, with one-skeleta $\Gbb_n=\mathcal T_n^{(1)}$, such that:
\begin{enumerate}
\item $o\in V(\Gbb_n)$ for every $n$;
\item every boundary component of $\Sigma$ is a subcomplex of $\mathcal T_n$;
\item $\mesh(\Gbb_n)\to0$;
\item $(\Gbb_n)_{n\geq1}$ is universally stable for regular loop families.
\end{enumerate}
\end{proposition}

The proof of Proposition~\ref{prop:existence-universally-stable-sequence} is deferred to Section~\ref{sec:proof_prop_univ_stable}, because it requires several technical results from PL topology and would interrupt the flow of the proofs. Note that L\'evy proved that every finite family of piecewise-geodesic paths is carried by a piecewise-geodesic graph and that every embedded graph admits arbitrarily close piecewise-geodesic approximations preserving its facial combinatorics and approximately preserving its face areas; see \cite[Propositions 1.4.8 and 1.4.10]{Lev10}. Proposition~\ref{prop:existence-universally-stable-sequence} is a sequential and simultaneous version adapted to scaling limits: it constructs a single vanishing-mesh sequence of graphs along which every finite regular loop family can be approximated while preserving its marked ambient ribbon type and the areas of its complementary components.

\begin{proposition}\label{prop:Levy-stable-existence}
Every finite regular based loop family admits a stable graph approximation.
\end{proposition}

\begin{proof}
It follows from Definition~\ref{def:universally-stable-graph-sequence}, and Proposition~\ref{prop:existence-universally-stable-sequence}. 
\end{proof}

\subsection{Regular convergence of holonomy processes}

\begin{definition}\label{def:regular-process-convergence}
Let $(\Gbb_n)$ be universally stable. Let, for each $n$, $H_n$ be a holonomy process on $\Gbb_n$\footnote{Which we recall is just a random element of $\Hom(L_o(\Gbb_n),G)$.}. Let $H$ be a holonomy process on $\Sigma$. We say that $H_n$ converges \emph{regularly} to $H$ if, for every finite regular based loop family $L=(\ell_1,\ldots,\ell_r)$ and every stable graph approximation $L_n=(\ell_{1,n},\ldots,\ell_{r,n})$ carried by $(\Gbb_n)$,
\[
 \bigl(H_n(\ell_{1,n}),\ldots,H_n(\ell_{r,n})\bigr)
 \Longrightarrow
 \bigl(H_{\ell_1},\ldots,H_{\ell_r}\bigr)
\]
as probability measures on $G^r$.
\end{definition}

\begin{lemma}\label{lem:invariant-tests-determine-weak-convergence}
Let $\nu_n$ and $\nu$ be probability measures on $G^r$ invariant under diagonal conjugation. If $\int F\,d\nu_n\rightarrow\int F\,d\nu$ for every $F\in C(G^r)^{\operatorname{Ad}G}$, then $\nu_n\Rightarrow\nu$.
\end{lemma}

\begin{proof}
For $\varphi\in C(G^r)$, define
\[
\varphi^\#(x_1,\ldots,x_r) = \int_G\varphi(gx_1g^{-1},\ldots,gx_rg^{-1})\,dg.
\]
Then $\varphi^\#\in C(G^r)^{\operatorname{Ad}G}$ and invariance gives
\[
\int\varphi\,d\nu_n=\int\varphi^\#\,d\nu_n,\qquad\int\varphi\,d\nu=\int\varphi^\#\,d\nu.
\]
The conclusion follows.
\end{proof}

\begin{definition}\label{def:stochastic-d1-continuity}
A holonomy process $H$ on $\Sigma$ is stochastically $d_1$-continuous if
\[
 d_1(\ell_m,\ell)\longrightarrow0
 \quad\Longrightarrow\quad
 H_{\ell_m}\longrightarrow H_\ell
 \quad\text{in probability}.
\]
\end{definition}

\begin{proposition}\label{prop:Levy-stoch-d1-continuity}
Let $(\mu_t)_{t\ge0}$ be an admissible conjugation-invariant L\'evy semigroup on $G$, and let $H^\mu$ be the holonomy process on $(\Sigma,o)$ induced by its associated regular Markovian holonomy field. Then $H^\mu$ is stochastically $d_1$-continuous.
\end{proposition}

\begin{proof}
By \cite[Theorem~4.3.1]{Lev10}, an admissible L\'evy process gives rise to a regular Markovian holonomy field, whose restriction to based loops is the holonomy process $H^\mu$. Stochastic continuity is the first regularity requirement in \cite[Definition~3.1.3]{Lev10}. Immediately after that definition, L\'evy notes that, because $G$ is compact and $d_G$ is bounded, the stated $L^1$ convergence is equivalent to convergence in probability. Finally, Propositions~1.2.14 and~1.2.16 of \cite{Lev10} show that the metric $d_1$ induces the path topology used there. This proves the assertion.
\end{proof}

\begin{lemma}[Uniqueness of the $d_1$-continuous extension]\label{lem:unique-d1-extension}
Let $H$ and $\widetilde H$ be stochastically $d_1$-continuous holonomy processes. If their finite-dimensional distributions agree on every finite regular based loop family, then they agree on every finite rectifiable based loop family.
\end{lemma}

\begin{proof}
For $L=(\ell_1,\ldots,\ell_r)$, choose regular families $L^{(k)}$ as in Lemma~\ref{lem:regular-density-d1}. Stochastic continuity gives convergence in probability of both holonomy vectors along $L^{(k)}$ to their respective vectors along $L$. Since the two laws agree for every $k$, their weak limits agree.
\end{proof}

\subsection{Proof of the universality theorem}

Recall that
\[
 a_t(\lambda)
 =
 \frac1{d_\lambda}\langle Q_t,\chi_\lambda\rangle.
\]

\begin{lemma}\label{lem:topological-coefficient-marked-invariance}
Let $L$ and $L'$ have the same marked ambient ribbon type, witnessed by $h$. Let $C_1,\ldots,C_m$ be the cut components for $L$, put $C'_j=h(C_j)$, and transport a labelling by
\[
 \Lambda'(C'_j)=\Lambda(C_j).
\]
Then, for every $F\in C(G^r)^{\operatorname{Ad}G}$,
\[
 \widehat W_{F,L'}(\Lambda')
 =
 \widehat W_{F,L}(\Lambda).
\]
\end{lemma}

\begin{proof}
Orient the cut edges of $L$ and transport these orientations by $h$. The induced bijection of cut edges identifies the Haar product spaces. Because $h$ preserves the marked loop words, the words $h_{\ell_i}$ in the cut-edge variables are identical under this identification. It also identifies, with their orientations, every marked boundary cycle of every cut component; original boundary components are unchanged because $h$ fixes $\partial\Sigma$ pointwise. Finally, $\chi(C'_j)=\chi(C_j)$. Thus the two defining Haar integrals are identical after relabelling variables. If different auxiliary edge orientations are chosen, the same conclusion follows from invariance of Haar measure under inversion.
\end{proof}

\begin{lemma}\label{lem:topological-coefficient-polynomial-bound}
Fix a regular loop family $L$ and $F\in L^\infty(G^r)$. There exist constants $M=M(L)$ and $C=C(L,\|F\|_\infty)$ such that
\[
 |\widehat W_{F,L}(\Lambda)|
 \le
 C\prod_{j=1}^m d_{\Lambda(C_j)}^M
\]
for every labelling $\Lambda$.
\end{lemma}

\begin{proof}
In the definition of $\widehat W_{F,L}$, use $|\chi_\lambda(g)|\le d_\lambda$. There are finitely many boundary-character factors, and the powers $d_{\Lambda(C_j)}^{\chi(C_j)}$ are bounded above by fixed nonnegative powers of the corresponding dimensions. Absorbing all these powers into a common exponent $M$ gives the estimate.
\end{proof}

\begin{lemma}\label{lem:levy-symbol-summability}
Let $\psi:\widehat G\to\C$ be the symbol of an admissible L\'evy process. For every $s>0$ and $M\ge0$,
\begin{equation}\label{eq:bound-symbol-exp}
 \sum_{\lambda\in\widehat G}
 d_\lambda^M e^{-s\Re\psi(\lambda)}<\infty.
\end{equation}
\end{lemma}

\begin{proof}
Let $Q_t^\psi$ be the density of the process at time $t>0$. Since it is continuous, it belongs to $L^2(G)$. Its scalar Fourier coefficient at $\lambda$ is $e^{-t\psi(\lambda)}$, so Plancherel gives
\[
 \|Q_t^\psi\|_{L^2(G)}^2
 =
 \sum_{\lambda\in\widehat G}
 d_\lambda^2 e^{-2t\Re\psi(\lambda)}<\infty.
\]
Hence
\[
 \sum_{\lambda\in\widehat G}d_\lambda^2e^{-r\Re\psi(\lambda)}<\infty
\]
for every $r>0$. This immediately proves the claim when $M\leq2$.

Assume $M>2$ and put $r=s/(M-2)$. Convergence of the preceding series implies
\[
 C_r:=\sup_{\lambda\in\widehat G}
 d_\lambda^2e^{-r\Re\psi(\lambda)}<\infty.
\]
Therefore
\[
 d_\lambda^{M-2}e^{-\frac{s}{2}\Re\psi(\lambda)}
 =
 \left(d_\lambda^2e^{-r\Re\psi(\lambda)}\right)^{(M-2)/2}
 \leq C_r^{(M-2)/2}.
\]
Multiplying by $d_\lambda^2e^{-\frac{s}{2}\Re\psi(\lambda)}$ and summing over $\lambda$ gives
\[
 \sum_{\lambda\in\widehat G}d_\lambda^M e^{-s\Re\psi(\lambda)}
 \leq
 C_r^{(M-2)/2}
 \sum_{\lambda\in\widehat G}d_\lambda^2e^{-\frac{s}{2}\Re\psi(\lambda)}<\infty.
\]
\end{proof}
\begin{lemma}\label{lem:triangular-products}
Fix $\lambda\in\widehat G$. Suppose that $t_{n,k}>0$ satisfy
\[
 \max_k t_{n,k}\xrightarrow[n\to\infty]{}0,
 \qquad
 \sum_k t_{n,k}\xrightarrow[n\to\infty]{} T<\infty.
\]
Under assumption (A2),
\[
 \prod_k a_{t_{n,k}}(\lambda)
 \longrightarrow
 e^{-T\psi(\lambda)}.
\]
\end{lemma}

\begin{proof}
For fixed $\lambda$, assumption (A2) can be written
\[
 a_t(\lambda)
 =
 e^{-t\psi(\lambda)}\bigl(1+t\varepsilon_\lambda(t)\bigr),
 \qquad
 \varepsilon_\lambda(t)\longrightarrow0.
\]
Put $u_{n,k}=t_{n,k}\varepsilon_\lambda(t_{n,k})$. Then
\[
 \sum_k|u_{n,k}|
 \le
 \left(\sum_k t_{n,k}\right)
 \sup_{0<t\le\max_k t_{n,k}}|\varepsilon_\lambda(t)|
 \longrightarrow0.
\]
Using
\[
 \left|\prod_k(1+u_{n,k})-1\right|
 \le
 \exp\left(\sum_k|u_{n,k}|\right)-1,
\]
we obtain $\prod_k(1+u_{n,k})\to1$. The exponential factors multiply to
\[
 \exp\left(-\psi(\lambda)\sum_k t_{n,k}\right)
 \longrightarrow e^{-T\psi(\lambda)}.
\]
\end{proof}

\begin{proof}[Proof of Theorem~\ref{thm:universality}]
Let $L=(\ell_1,\ldots,\ell_r)$ be regular and let $L_n$ be an arbitrary stable approximation carried by $(\Gbb_n)$. Let $\nu^\psi$ be the corresponding law of the holonomy vector of $H^\psi$. We first work with the unnormalized lattice measures; the argument below will show that their partition functions are strictly positive for all sufficiently large $n$. Since the normalized lattice laws and $\nu^\psi$ are invariant under diagonal conjugation, Lemma~\ref{lem:invariant-tests-determine-weak-convergence} shows that it suffices to test against
\[
 F\in C(G^r)^{\operatorname{Ad}G}.
\]

Let $C_1,\ldots,C_m$ be the cut components for $L$. For all
sufficiently large $n$, choose the homeomorphism $h_n$ in
Definition~\ref{def:Levy-stable-graph-approximation} and write
$C_{j,n}=h_n(C_j)$. Every $C_j$ is a nonempty open subset of $\Sigma$ and
hence has positive area. By stability,
\[
 |C_{j,n}|\longrightarrow|C_j|,
\]
and by Lemma~\ref{lem:mesh-implies-face-area},
\[
 \delta_n:=\max_{f\in F(\Gbb_n)}|f|\longrightarrow0.
\]

Choose $\eta_j>0$ such that $|C_{j,n}|\geq\eta_j$ for all sufficiently
large $n$. Let $M_0$ be supplied by
Lemma~\ref{lem:topological-coefficient-polynomial-bound} and put
$M=\max\{M_0,2\}$. For every $j$, apply (A3) with $T=\eta_j$ and this
value of $M$. Denote the resulting small-time threshold by $t_j$ and set
\[
 b_j(\lambda)
 :=\sup_{0<t\leq t_j}|a_t(\lambda)|^{\eta_j/t}.
\]
Then
\begin{equation}\label{eq:universality-single-plaquette-envelope}
 \sum_{\lambda\in\widehat G}
 d_\lambda^M b_j(\lambda)<\infty.
\end{equation}
Moreover, (A1) implies $0\leq b_j(\lambda)\leq1$. Since
$\delta_n\to0$, all face areas are at most $\min_jt_j$ for all
sufficiently large $n$. For such $n$,
\begin{align}
 \prod_{f\subset C_{j,n}}
 |a_{|f|}(\lambda)|
 &=\prod_{f\subset C_{j,n}}
 \left(
 |a_{|f|}(\lambda)|^{\eta_j/|f|}
 \right)^{|f|/\eta_j}\notag\\
 &\leq b_j(\lambda)^{|C_{j,n}|/\eta_j}
 \leq b_j(\lambda).
 \label{eq:component-product-envelope}
\end{align}
Writing $g_j$ for the genus of $C_j$, we have $2-2g_j\leq2\leq M$, and hence
\[
 \sum_{\lambda\in\widehat G}
 d_\lambda^{2-2g_j}
 \prod_{f\subset C_{j,n}}|a_{|f|}(\lambda)|
 \leq
 \sum_{\lambda\in\widehat G}d_\lambda^M b_j(\lambda)<\infty.
\]
Thus the componentwise hypothesis
\eqref{eq:component-spectral-summability} holds for all sufficiently
large $n$. Identify labels on the components through $h_n$. By
Remark~\ref{rem:topological-coefficient-subdivision}, the topological
coefficient may be computed using the minimal loop graph. Theorem~\ref{thm:Fourier_Wilsonloop} and
Lemma~\ref{lem:topological-coefficient-marked-invariance} therefore give
the non-normalized integral
\[
 \mathcal E_n(F)
 =
 \sum_{\Lambda:\{C_1,\ldots,C_m\}\to\widehat G}
 \left(
 \prod_{j=1}^m
 \prod_{f\subset C_{j,n}}
 a_{|f|}(\Lambda(C_j))
 \right)
 \widehat W_{F,L}(\Lambda).
\]
For a fixed label $\Lambda$, the face areas in $C_{j,n}$ have maximum at
most $\delta_n$ and sum to $|C_{j,n}|$. Since $|C_{j,n}|\to |C_j|$,
Lemma~\ref{lem:triangular-products} therefore gives
\[
 \prod_{f\subset C_{j,n}}
 a_{|f|}(\Lambda(C_j))
 \longrightarrow
 e^{-|C_j|\psi(\Lambda(C_j))}.
\]

For domination, combine
Lemma~\ref{lem:topological-coefficient-polynomial-bound} with
\eqref{eq:component-product-envelope}. Since $M\geq M_0$,
\[
\left|
 \left(
 \prod_{j=1}^m\prod_{f\subset C_{j,n}}
 a_{|f|}(\Lambda(C_j))
 \right)
 \widehat W_{F,L}(\Lambda)
\right|
\leq
 C\prod_{j=1}^m
 d_{\Lambda(C_j)}^M b_j(\Lambda(C_j)).
\]
The right-hand side is summable over all labels by
\eqref{eq:universality-single-plaquette-envelope}. Dominated convergence
yields
\[
 \mathcal E_n(F)
 \longrightarrow
 \mathcal E^\psi(F)
 :=
 \sum_\Lambda
 \left(
 \prod_{j=1}^m
 e^{-|C_j|\psi(\Lambda(C_j))}
 \right)
 \widehat W_{F,L}(\Lambda).
\]
Applying the same argument to $F=1$ gives
\[
 Z_n=\mathcal E_n(1)
 \longrightarrow
 Z^\psi=\mathcal E^\psi(1).
\]
By admissibility, every L\'evy density $Q_t^\psi$ is strictly positive. In the rigorous boundary-condition formula~\eqref{eq:DS-rigorous}, the limiting partition function $Z^\psi$ is therefore the integral of a strictly positive continuous function against a product of probability measures; hence $Z^\psi>0$, both on closed surfaces and in the presence of prescribed boundary conjugacy classes. It follows that $Z_n>0$ for all sufficiently large $n$. For those $n$, let $\nu_n$ be the law of the holonomy vector under the normalized lattice gauge measure. Then
\[
 \int F\,d\nu_n
 =
 \frac{\mathcal E_n(F)}{Z_n}
 \longrightarrow
 \frac{\mathcal E^\psi(F)}{Z^\psi}.
\]

For the admissible L\'evy semigroup with symbol $\psi$, the Fourier coefficient of its time-$t$ density is exactly $e^{-t\psi(\lambda)}$. Lemma~\ref{lem:levy-symbol-summability} gives the graph-level product summability needed to apply the state-sum formula to these densities. The graph measures are precisely those of \cite[Definition~4.3.2]{Lev10}; Proposition~4.3.10 there proves that they form a discrete Markovian holonomy field. Propositions~4.3.11 and~4.3.15 prove its regularity, and Theorem~3.2.9 extends it uniquely to a regular Markovian holonomy field. This construction is summarized in \cite[Theorem~4.3.1]{Lev10}; its restriction to the based loops of the fixed surface is precisely the holonomy process $H^\psi$ used here. Applying the same state-sum formula to these graph measures therefore identifies
\[
 \frac{\mathcal E^\psi(F)}{Z^\psi}
 =
 \mathbb E\left[
 F(H^\psi_{\ell_1},\ldots,H^\psi_{\ell_r})
 \right].
\]
Thus $\nu_n\Rightarrow\nu^\psi$. Since the stable approximation was arbitrary, the lattice holonomy processes converge regularly to $H^\psi$. Lemma~\ref{lem:unique-d1-extension} shows that this regular limit determines the finite-dimensional distributions on all rectifiable based loops.
\end{proof}

\section{Existence of universally stable graph sequences}\label{sec:proof_prop_univ_stable}

This section is devoted to the proof of Proposition~\ref{prop:existence-universally-stable-sequence}. It goes as follows: first, by a separability argument, we choose a countable collection of regular loop families which is dense, in the $d_1$-topology, among all regular loop families and which approximates them without changing their embedded topology. Second, we successively insert these countably many families into a growing sequence of finite embedded graphs. A small general-position perturbation is used before each insertion to ensure that the new family meets the graph already constructed only at finitely many points. Third, each resulting finite graph is completed to an arbitrarily fine triangulation of the surface. Enumerating every element of the countable dense collection infinitely often then gives a diagonal argument: every prescribed regular loop family is approximated, with arbitrarily small error, at arbitrarily late stages of the same graph sequence.

A \emph{marked abstract loop type} consists of a finite connected abstract ribbon graph with a distinguished vertex, together with an ordered finite family of closed oriented edge-words based at that vertex. An embedding of such a type in $\operatorname{int}(\Sigma)$, sending the distinguished vertex to $o$, determines a regular based loop family.

\begin{lemma}\label{lem:ribbon-neighbourhood-isotopy}
Let $L$ be a finite regular based loop family, and let $R\Subset\operatorname{int}(\Sigma)$ be a closed ribbon neighbourhood of $\Gamma(L)$, decomposed into pairwise disjoint vertex disks and edge strips. Let $L'$ be a regular loop family of the same marked abstract loop type such that:
\begin{enumerate}
\item the vertex of $L'$ corresponding to a vertex $v$ of $\Gamma(L)$ lies in the corresponding vertex disk;
\item in every edge strip, the corresponding edge of $L'$ is a single properly embedded arc joining the two prescribed attaching intervals;
\item no edge of $L'$ enters a vertex disk or edge strip to which it is not incident, and the cyclic order at every vertex is the same as for $L$.
\end{enumerate}
Then $L$ and $L'$ have the same marked ambient ribbon type. More precisely, there is an ambient isotopy of $\Sigma$, fixing $o$ and $\partial\Sigma$ pointwise and supported in $\operatorname{int}(R)$, whose final homeomorphism sends every marked loop of $L$ to the corresponding marked loop of $L'$.
\end{lemma}

\begin{proof}
Informally, the vertex disks allow the vertices and incident arms to move without crossing, while each edge strip is a rectangle in which there is no topological obstruction to deforming one properly embedded arc into another with the same endpoints.
The assertion reduces to the relative unknotting of finitely many tame arcs in disks. Piecewise-smooth arcs are tame, and may be straightened to PL arcs inside arbitrarily small tubular neighbourhoods. We spell out the reduction and use the relative disk theorem for pairs \cite[Theorem~4.20]{RourkeSandersonPL} in the case of the standard pair $(I^2,I^1)$.

First work in a vertex disk. Shrink it slightly so that the portions of both graphs in the disk are finite embedded stars, meeting the new boundary circle once for every incident half-edge. Since the two cyclic orders agree, an orientation-preserving isotopy of the boundary circle, supported in the attaching intervals, matches the corresponding intersection points. Extend this isotopy to a collar of the boundary and, when the disk corresponds to the base point, keep $o$ fixed. After moving the central vertex along a short embedded arc, choose pairwise disjoint narrow sector-disks around the incident arms, disjoint away from a small disk about the central vertex. In each sector-disk the two corresponding arms are tame properly embedded arcs with the same endpoints. The $(2,1)$ case of \cite[Theorem~4.20]{RourkeSandersonPL}, applied relative to the boundary of the sector, supplies an ambient isotopy carrying one arm to the other. Performing these finitely many isotopies successively carries the original star to the target star, fixes the outer boundary of the vertex disk, and may be arranged to make the two graphs agree near every attaching interval.

Now consider an edge strip. After the preceding vertex-disk isotopies, the original and target edges agree near the two attaching intervals. They are tame properly embedded arcs in a rectangle with the same endpoints. Applying again the $(2,1)$ case of \cite[Theorem~4.20]{RourkeSandersonPL}, relative to the boundary of the rectangle, gives an isotopy of the strip carrying the original core to the target arc and fixed near the whole boundary of the strip.

The vertex-disk and edge-strip isotopies have disjoint supports except on regions where they are already stationary, so they concatenate to an isotopy of $R$ relative to $\partial R$ and to $o$. Extending it by the identity on $\Sigma\setminus R$ proves the result. The construction respects the marked edge words, hence sends every marked loop of $L$ to the corresponding marked loop of $L'$.
\end{proof}

\begin{lemma}\label{lem:countable-dense-regular-families}
There exists a countable collection $\mathscr D$ of finite regular based loop families with the following property. For every finite regular based loop family $L=(\ell_1,\ldots,\ell_r)$ and every $\varepsilon>0$, there exists $L'=(\ell'_1,\ldots,\ell'_r)\in\mathscr D$ such that:
\begin{enumerate}
\item $L'$ has the same marked ambient ribbon type as $L$;
\item $\max_i d_1(\ell'_i,\ell_i)<\varepsilon;$
\item the marked ambient equivalence can be realized by an ambient isotopy fixing $o$ and $\partial\Sigma$ pointwise and supported in an open set of area smaller than $\varepsilon$.
\end{enumerate}
Hence, corresponding complementary components have areas differing by less than $\varepsilon$.
\end{lemma}

\begin{proof}
There are only countably many marked abstract loop types. Fix one, denoted by $\tau$, and let $K_\tau$ be its underlying finite graph. Consider the space $\operatorname{Emb}_o^1(K_\tau,\operatorname{int}(\Sigma))$ of edgewise $C^1$ embeddings which send the distinguished vertex to $o$ and induce the prescribed cyclic orders, equipped with the edgewise $C^1$ topology. This space is second countable and therefore separable: after embedding $\Sigma$ smoothly in a Euclidean space and parametrizing the finitely many abstract edges by compact intervals, it is a subspace of a finite product of separable $C^1$ function spaces. Choose a countable dense subset $\mathscr E_\tau$, and let $\mathscr D$ be the union of the associated loop families over all marked types $\tau$.

Fix $L$ and $\varepsilon>0$. By rounding its finitely many edge corners inside pairwise disjoint coordinate disks of total area less than $\varepsilon/3$, replace $L$ by a family $L^{\mathrm{sm}}$ of the same marked ambient ribbon type whose edges are $C^1$ and such that
\[
 \max_i d_1(\ell_i^{\mathrm{sm}},\ell_i)<\frac{\varepsilon}{3}.
\]
The rounding is realized by an ambient isotopy fixing $o$ and $\partial\Sigma$ pointwise and supported in those disks. Choose a closed ribbon neighbourhood $R\Subset\operatorname{int}(\Sigma)$ of $\Gamma(L^{\mathrm{sm}})$, made of vertex disks and edge strips, with $|R|<\varepsilon/3$. By density, choose an embedding $\iota'\in\mathscr E_\tau$ sufficiently close that its vertices and edges lie in the prescribed disks and strips with the same incidences and cyclic orders, and such that the associated loop family $L'$ satisfies
\[
 \max_i d_1(\ell_i',\ell_i^{\mathrm{sm}})<\frac{\varepsilon}{3}.
\]
The triangle inequality gives the required $d_1$ estimate. Composing the corner-rounding isotopy with the isotopy furnished by Lemma~\ref{lem:ribbon-neighbourhood-isotopy} gives the marked ambient equivalence between $L$ and $L'$. Its support is contained in an open set of area less than $2\varepsilon/3<\varepsilon$. If $C$ is a complementary component of $\Gamma(L)$, then $C\triangle h(C)$ is contained in the support of the resulting homeomorphism $h$; hence corresponding complementary components have areas differing by less than $\varepsilon$.
\end{proof}

\begin{lemma}\label{lem:fine-relative-triangulation}
Let $K$ be a finite piecewise-smooth graph embedded in $\Sigma$, containing $o$ and carrying $\partial\Sigma$. For every $\varepsilon>0$, there exists a smooth triangulation $\mathcal T$ of $\Sigma$ such that:
\begin{enumerate}
\item $K$ and $\partial\Sigma$ are subcomplexes of $\mathcal T$;
\item every edge of $\mathcal T$ has length at most $\varepsilon$;
\item every two-simplex of $\mathcal T$ has diameter at most $\varepsilon$.
\end{enumerate}
\end{lemma}

\begin{proof}
Subdivide $K$ at every corner and at the point $o$. Choose a smooth closed regular neighborhood $N$ of $K$, relative to $\partial\Sigma$, with rounded corners, consisting of vertex disks, strips around interior edges, and collar rectangles around the edges contained in $\partial\Sigma$. We take $N$ to contain a collar of the whole boundary. Triangulate each interior strip through a product parametrization $[0,1]\times[-1,1]$, with its core and boundary sides as subcomplexes; triangulate each boundary collar rectangle through $[0,1]\times[0,1]$ in the same way; and triangulate every vertex disk relative to the incident core arcs and boundary arcs. Choose a common finite subdivision on every attaching interval before triangulating the pieces. The resulting triangulations therefore glue to a finite smooth triangulation of $N$ for which $K$, $\partial\Sigma$ and the interior frontier $F:=\operatorname{Fr}_\Sigma(N)$ are subcomplexes.

The closure $M=\overline{\Sigma\setminus N}$ is a compact smooth surface with boundary $F$. By the relative smooth triangulation theorem \cite[Theorem~10.6]{Mun66}, the prescribed triangulation of $F$ extends to a smooth triangulation of $M$. Gluing gives a smooth triangulation of $\Sigma$ containing $K$ and $\partial\Sigma$ as subcomplexes.

Each simplex map in this finite triangulation is Lipschitz on its compact Euclidean simplex. The Euclidean mesh of iterated barycentric subdivisions tends uniformly to zero; hence, after sufficiently many subdivisions, all Riemannian edge lengths and all Riemannian diameters of two-simplices are at most $\varepsilon$. The prescribed subcomplexes remain subcomplexes.
\end{proof}

The next lemma is the general-position step needed to enlarge the graph inductively. Its geometric content is simple in dimension two. Two one-dimensional objects in a two-dimensional surface are generically expected to meet at isolated points rather than along common arcs. The PL general-position theorem allows us to realize such a perturbation while keeping the base point fixed, making the perturbation arbitrarily small, and localizing it in a prescribed neighbourhood of the new loop family. Since the graphs involved are finite polyhedra, a compact zero-dimensional intersection is then a finite set.

\begin{lemma}\label{lem:relative-general-position-graphs}
Let $K$ be a finite piecewise-smooth graph embedded in $\Sigma$, containing $o$ and carrying $\partial\Sigma$, and let $L$ be a finite regular based loop family. For every $\varepsilon>0$, there exists a regular loop family $\widetilde L$ such that:
\begin{enumerate}
\item $\widetilde L$ has the same marked ambient ribbon type as $L$;
\item we have
\[
 \max_i d_1(\widetilde\ell_i,\ell_i)<\varepsilon;
\]
\item $\Gamma(\widetilde L)$ and $K$ have no common nontrivial arc and have finite intersection;
\item the marked ambient equivalence is realized by an ambient isotopy fixing $o$ and $\partial\Sigma$ pointwise and supported in an open set of area smaller than $\varepsilon$.
\end{enumerate}
Consequently, after subdivision at the intersection points and edge corners, $K\cup\Gamma(\widetilde L)$ is a finite embedded graph.
\end{lemma}

\begin{proof}
Choose closed ribbon neighbourhoods
\[
 \Gamma(L)\subset\operatorname{int}(R_0)
 \subset R_0\subset\operatorname{int}(R)
 \subset R\Subset\operatorname{int}(\Sigma)
\]
with $|R|<\varepsilon$. By Lemma~\ref{lem:fine-relative-triangulation}, choose a smooth triangulation in which $K$ is a subcomplex. After a sufficiently fine subdivision, approximate the marked edges of $\Gamma(L)$, inside the prescribed vertex disks and edge strips of $R_0$, by a marked PL embedded graph $P$ in the induced PL structure. We keep $o$ fixed, preserve all incidences and cyclic orders, and choose the approximation so close that the associated loop family $L^{\mathrm{PL}}$ satisfies
\[
 \max_i d_1(\ell_i^{\mathrm{PL}},\ell_i)<\frac{\varepsilon}{2}.
\]
The approximation is chosen in the ribbon position described in Lemma~\ref{lem:ribbon-neighbourhood-isotopy}; hence that lemma realizes it by an ambient isotopy supported in $\operatorname{int}(R_0)$, fixing $o$ and $\partial\Sigma$ pointwise.

Let $M=\operatorname{int}(R)$ with the induced PL structure. This is a PL surface without boundary. Choose a compact polyhedral neighbourhood $S$ of $P$ such that
\[
 P\subset\operatorname{int}(S)\subset S\subset\operatorname{int}(R_0).
\]
After refining the triangulation sufficiently, we may assume that $S$ is a finite subpolyhedron and that the union $Q$ of all simplices of the one-dimensional subcomplex $K$ which meet $S$ is contained in $\operatorname{int}(R)$. Then $Q$ is a finite subpolyhedron of $M$ and contains $K\cap S$.
Apply the general-position theorem for embeddings \cite[Theorem~5.3]{RourkeSandersonPL} in $M$ with
\[
 P^1=P,\qquad P_0=\{o\},\qquad Q^1=Q,
 \qquad m=2.
\]
Choose the size of the isotopy smaller than the distance from $P$ to $M\setminus\operatorname{int}(S)$. The theorem gives a compactly supported ambient PL isotopy of $M$, fixed at $o$, whose final homeomorphism $u$ satisfies
\[
 \dim\bigl(u(P\setminus\{o\})\cap Q\bigr)\le 1+1-2=0.
\]
The isotopy is supported away from $\partial R$, so it extends by the identity to an ambient isotopy of $\Sigma$, supported in $R$ and fixing $\partial\Sigma$ pointwise. Moreover, $u(P)\subset S$. Since $Q$ contains $K\cap S$, the intersection $u(P)\cap K$ is a compact zero-dimensional polyhedron together with the possible point $o$, and is therefore finite. In particular, $u(P)$ and $K$ have no common nontrivial arc.

For completeness, the proof of \cite[Theorem~5.3]{RourkeSandersonPL} produces the isotopy by moving finitely many vertices of a subdivision and extending conically over their stars. The vertex displacements may be chosen arbitrarily small. For each of the finitely many marked PL paths, both the uniform distance and the lengths vary continuously with these vertex positions. Proposition~1.2.16 of \cite{Lev10} therefore implies that the final perturbation may be chosen so that
\[
 \max_i d_1\bigl(u(\ell_i^{\mathrm{PL}}),\ell_i^{\mathrm{PL}}\bigr)
 <\frac{\varepsilon}{2}.
\]
Set $\widetilde L=u(L^{\mathrm{PL}})$. The triangle inequality gives the required $d_1$ estimate. The composition of the two ambient isotopies witnesses that $\widetilde L$ has the same marked ambient ribbon type as $L$; it fixes $o$ and $\partial\Sigma$ pointwise and is supported in $R$, whose area is smaller than $\varepsilon$. After subdividing at the finitely many intersection points and at all corners, $K\cup\Gamma(\widetilde L)$ is a finite embedded graph.
\end{proof}

Before giving the construction, let us describe the diagonal argument. Enumerate the countable dense family repeatedly, so that every one of its elements is presented infinitely many times. At stage $n$, we perturb the $n$-th family slightly so that it has only finitely many intersections with what has already been constructed, and we permanently add it to the graph $K_n$. Hence, once a loop family has been inserted, it remains carried by every subsequent $K_n$.

Now fix an arbitrary regular family $L$. Choose countable-model families $L^{(j)}$ approaching $L$ with error tending to zero. Since every model occurs infinitely often in the enumeration, we may choose increasing stages $m_j$ at which $L^{(j)}$ is inserted, and moreover choose them so late that the additional general-position perturbation is also arbitrarily small. Between stages $m_j$ and $m_{j+1}$, we simply use the copy inserted at stage $m_j$ as the approximation of $L$. This produces a stable approximation along the \emph{entire} graph sequence, not merely along a subsequence.

\begin{proof}[Proof of Proposition~\ref{prop:existence-universally-stable-sequence}]
Let $\mathscr D$ be given by Lemma~\ref{lem:countable-dense-regular-families}. Enumerate it with repetitions,
\[
 \mathscr D=\{\overline L^{\,1},\overline L^{\,2},\ldots\},
\]
so that every element occurs infinitely often, and fix $\delta_n\downarrow0$.

We construct finite embedded graphs
\[
 K_0\subset K_1\subset K_2\subset\cdots
\]
and perturbed families $\widetilde L^{\,n}$. Let $K_0$ be the union of $o$ with a fixed finite graph structure on $\partial\Sigma$. Given $K_{n-1}$, apply Lemma~\ref{lem:relative-general-position-graphs} to $K_{n-1}$, $\overline L^{\,n}$ and $\delta_n$. Define
\[
 K_n:=K_{n-1}\cup\Gamma(\widetilde L^{\,n}),
\]
subdividing at every intersection point, vertex and corner. This is a finite embedded graph. Apply Lemma~\ref{lem:fine-relative-triangulation} to $K_n$ with $\varepsilon=n^{-1}$, call the triangulation $\mathcal T_n$, and put $\Gbb_n=\mathcal T_n^{(1)}$. Then
\[
 \mesh(\Gbb_n)\le n^{-1},
\]
and the first three assertions hold.

Fix a regular family $L=(\ell_1,\ldots,\ell_r)$. For every $j\ge1$, choose from $\mathscr D$ a family $L^{(j)}$ and an ambient isotopy $g_j$ as in Lemma~\ref{lem:countable-dense-regular-families}, with error $j^{-1}$. Thus $g_j$ witnesses marked ambient equivalence between $L$ and $L^{(j)}$, the $d_1$ error is less than $j^{-1}$, and corresponding component areas differ by less than $j^{-1}$.

Choose strictly increasing indices $m_j$ such that
\[
 \overline L^{\,m_j}=L^{(j)},
 \qquad
 \delta_{m_j}<j^{-1}.
\]
Let $q_j$ be the ambient isotopy supplied by the general-position lemma from $L^{(j)}$ to $\widetilde L^{\,m_j}$. For $n<m_1$, let $L_n$ consist of $r$ copies of the constant loop at $o$. If $m_j\le n<m_{j+1}$, put
\[
 L_n:=\widetilde L^{\,m_j}.
\]
It is carried by $\Gbb_n$, because
\[
 \Gamma(\widetilde L^{\,m_j})
 \subset K_{m_j}\subset K_n\subset\Gbb_n.
\]
Subdivided edges are read as concatenations of subedges.

For $m_j\le n<m_{j+1}$,
\[
\begin{aligned}
 \max_i d_1(\ell_{i,n},\ell_i)
 &\le
 \max_i d_1(\widetilde\ell_i^{\,m_j},\ell_i^{(j)})
 +
 \max_i d_1(\ell_i^{(j)},\ell_i) \\
 &<\delta_{m_j}+j^{-1}<2j^{-1}.
\end{aligned}
\]
The composition $h_n=q_j\circ g_j$ witnesses that $L_n$ and $L$ have the same marked ambient ribbon type. If $C_1,\ldots,C_m$ are the components for $L$ and $C_{k,n}=h_n(C_k)$, then
\[
 \bigl||C_{k,n}|-|C_k|\bigr|
 <j^{-1}+\delta_{m_j}\longrightarrow0.
\]
Thus $(L_n)$ is a stable approximation of $L$. Since $L$ was arbitrary, $(\Gbb_n)$ is universally stable.
\end{proof}

\bibliographystyle{alpha}
\bibliography{Biblio}

@article {BlauThom92,
    AUTHOR = {Blau, Matthias and Thompson, George},
     TITLE = {Quantum {Y}ang-{M}ills theory on arbitrary surfaces},
   JOURNAL = {Internat. J. Modern Phys. A},
  FJOURNAL = {International Journal of Modern Physics A. Particles and
              Fields. Gravitation. Cosmology. Astrophysics. Accelerator
              Physics},
    VOLUME = {7},
      YEAR = {1992},
    NUMBER = {16},
     PAGES = {3781--3806},
      ISSN = {0217-751X,1793-656X},
   MRCLASS = {81T40 (58D30 58Z05 81S40 81T13)},
  MRNUMBER = {1169172},
MRREVIEWER = {Sylvie\ Paycha},
       DOI = {10.1142/S0217751X9200168X},
       URL = {https://doi.org/10.1142/S0217751X9200168X},
}

@misc{BCDRT26,
      title={The {Y}ang--{M}ills measure on surfaces via Morse theory}, 
      author={Chhaibi, Reda and  Dang, Nguyen Viet and Guedes Bonthonneau, Yannick and  Rivière, Gabriel and Tô, Tat Dat },
      year={2026},
      eprint={2607.24640},
      archivePrefix={arXiv},
      primaryClass={math.PR},
      url={https://arxiv.org/abs/2607.24640},
      note={arXiv:2607.24640}
}

@article {CCHS22,
    AUTHOR = {Chandra, Ajay and Chevyrev, Ilya and Hairer, Martin and Shen,
              Hao},
     TITLE = {Langevin dynamic for the 2{D} {Y}ang-{M}ills measure},
   JOURNAL = {Publ. Math. Inst. Hautes \'{E}tudes Sci.},
  FJOURNAL = {Publications Math\'{e}matiques. Institut de Hautes \'{E}tudes
              Scientifiques},
    VOLUME = {136},
      YEAR = {2022},
     PAGES = {1--147},
      ISSN = {0073-8301,1618-1913},
   MRCLASS = {81T08 (58J35 60H15 60J25 81T13)},
  MRNUMBER = {4517645},
MRREVIEWER = {Jiang\ Lun\ Wu},
       DOI = {10.1007/s10240-022-00132-0},
       URL = {https://doi.org/10.1007/s10240-022-00132-0},
}

@article {Chev19,
    AUTHOR = {Chevyrev, Ilya},
     TITLE = {Yang-{M}ills measure on the two-dimensional torus as a random
              distribution},
   JOURNAL = {Comm. Math. Phys.},
  FJOURNAL = {Communications in Mathematical Physics},
    VOLUME = {372},
      YEAR = {2019},
    NUMBER = {3},
     PAGES = {1027--1058},
      ISSN = {0010-3616,1432-0916},
   MRCLASS = {53C05 (53C29 60H30)},
  MRNUMBER = {4034782},
       DOI = {10.1007/s00220-019-03567-5},
       URL = {https://doi.org/10.1007/s00220-019-03567-5},
}

@article {CKM,
    AUTHOR = {Chevyrev, Ilya and Klose, Tom and  Mohamed, Abdulwahab},
     TITLE = {A PDE approach to the 2D Yang-Mills measure},
       year={2026},
      eprint={2607.22236},
      archivePrefix={arXiv},
      primaryClass={math.AP},
      url={https://arxiv.org/abs/2607.22236},
      note={arXiv:2607.22236},
}

@article {ChevGar25,
    AUTHOR = {Chevyrev, Ilya and Garban, Christophe},
     TITLE = {Villain action in lattice gauge theory},
   JOURNAL = {J. Stat. Phys.},
  FJOURNAL = {Journal of Statistical Physics},
    VOLUME = {192},
      YEAR = {2025},
    NUMBER = {3},
     PAGES = {Paper No. 38, 15},
      ISSN = {0022-4715,1572-9613},
   MRCLASS = {60D05 (05C20 81T13 81T25)},
  MRNUMBER = {4876989},
MRREVIEWER = {Vladimir\ Balan},
       DOI = {10.1007/s10955-025-03420-1},
       URL = {https://doi.org/10.1007/s10955-025-03420-1},
}

@article {ChevShen26,
    AUTHOR = {Chevyrev, Ilya and Shen, Hao},
     TITLE = {Invariant measure and universality of the 2{D} {Y}ang-{M}ills
              {L}angevin dynamic},
   JOURNAL = {Comm. Pure Appl. Math.},
  FJOURNAL = {Communications on Pure and Applied Mathematics},
    VOLUME = {79},
      YEAR = {2026},
    NUMBER = {8},
     PAGES = {1973--2102},
      ISSN = {0010-3640,1097-0312},
   MRCLASS = {60H15 (60L30 81T13 81T27)},
  MRNUMBER = {5084118},
       DOI = {10.1002/cpa.70043},
       URL = {https://doi.org/10.1002/cpa.70043},
}

@misc{Dah26,
      title={{L}arge {N} limit of {W}ilson {L}oops on orientable closed surfaces in the light of {K}oike-{S}chur-{W}eyl duality and {S}pin Networks}, 
      author={Antoine Dahlqvist},
      year={2026},
      eprint={2603.11374},
      archivePrefix={arXiv},
      primaryClass={math.PR},
      url={https://arxiv.org/abs/2603.11374},
      note={arXiv:2603.11374},
}

@misc{DangNohra26,
      title={The {Y}ang--{M}ills measure on compact surfaces as a universal scaling limit of lattice gauge models}, 
      author={Nguyen Viet Dang and Elias Nohra},
      year={2026},
      eprint={2602.08591},
      archivePrefix={arXiv},
      primaryClass={math.PR},
      url={https://arxiv.org/abs/2602.08591},
      note={arXiv:2602.08591}
}

@misc{DangNohra26b,
      title={Semiclassical analysis for Yang--Mills random connections on compact surfaces}, 
      author={Nguyen Viet Dang and Elias Nohra},
      year={2026},
      eprint={2607.19037},
      archivePrefix={arXiv},
      primaryClass={math.PR},
      url={https://arxiv.org/abs/2607.19037}, 
      note={arXiv:2607.19037},
}

@article {Dri89,
    AUTHOR = {Driver, Bruce K.},
     TITLE = {Y{M{${}_2$}}: continuum expectations, lattice convergence, and
              lassos},
   JOURNAL = {Comm. Math. Phys.},
  FJOURNAL = {Communications in Mathematical Physics},
    VOLUME = {123},
      YEAR = {1989},
    NUMBER = {4},
     PAGES = {575--616},
      ISSN = {0010-3616,1432-0916},
   MRCLASS = {81T08 (81T13 81T25)},
  MRNUMBER = {1006295},
MRREVIEWER = {Edward\ P.\ Osipov},
       URL = {http://projecteuclid.org/euclid.cmp/1104178984},
}

@article {GKS89,
    AUTHOR = {Gross, Leonard and King, Christopher and Sengupta, Ambar},
     TITLE = {Two-dimensional {Y}ang-{M}ills theory via stochastic
              differential equations},
   JOURNAL = {Ann. Physics},
  FJOURNAL = {Annals of Physics},
    VOLUME = {194},
      YEAR = {1989},
    NUMBER = {1},
     PAGES = {65--112},
      ISSN = {0003-4916,1096-035X},
   MRCLASS = {81E08 (60G15 60H10 81E13 81E40)},
  MRNUMBER = {1015789},
MRREVIEWER = {Bruce\ K.\ Driver},
       DOI = {10.1016/0003-4916(89)90032-8},
       URL = {https://doi.org/10.1016/0003-4916(89)90032-8},
}

@article{Lem26b,
      title={Universal dualities for {W}ilson loops in lattice {Y}ang-{M}ills}, 
      author={Thibaut Lemoine},
      year={2026},
      eprint={2604.16252},
      archivePrefix={arXiv},
      primaryClass={math-ph},
      url={https://arxiv.org/abs/2604.16252},
      note={arXiv:2604.16252},
}

@misc{Lem26c,
      title={The heat-kernel master field on $\mathbb{Z}^d$ at strong coupling}, 
      author={Thibaut Lemoine},
      year={2026},
      eprint={2606.28945},
      archivePrefix={arXiv},
      primaryClass={math-ph},
      url={https://arxiv.org/abs/2606.28945}, 
      note={arXiv:2606.28945},
      
}

@article {Lev03,
    AUTHOR = {L{\'{e}}vy, Thierry},
     TITLE = {Yang-{M}ills measure on compact surfaces},
   JOURNAL = {Mem. Amer. Math. Soc.},
  FJOURNAL = {Memoirs of the American Mathematical Society},
    VOLUME = {166},
      YEAR = {2003},
    NUMBER = {790},
     PAGES = {xiv+122},
      ISSN = {0065-9266,1947-6221},
   MRCLASS = {58D20 (60F20 60G60 60H40 81T13 81T27)},
  MRNUMBER = {2006374},
MRREVIEWER = {Ambar\ N.\ Sengupta},
       DOI = {10.1090/memo/0790},
       URL = {https://doi.org/10.1090/memo/0790},
}

@article {Lev04,
    AUTHOR = {L{\'{e}}vy, Thierry},
     TITLE = {Wilson loops in the light of spin networks},
   JOURNAL = {J. Geom. Phys.},
  FJOURNAL = {Journal of Geometry and Physics},
    VOLUME = {52},
      YEAR = {2004},
    NUMBER = {4},
     PAGES = {382--397},
      ISSN = {0393-0440,1879-1662},
   MRCLASS = {81T13 (20C35 53C80 81R05)},
  MRNUMBER = {2098832},
MRREVIEWER = {Ambar\ N.\ Sengupta},
       DOI = {10.1016/j.geomphys.2004.04.003},
       URL = {https://doi.org/10.1016/j.geomphys.2004.04.003},
}

@article {Lev10,
    AUTHOR = {L{\'{e}}vy, Thierry},
     TITLE = {Two-dimensional {M}arkovian holonomy fields},
   JOURNAL = {Ast\'{e}risque},
  FJOURNAL = {Ast\'{e}risque},
    NUMBER = {329},
      YEAR = {2010},
     PAGES = {172},
      ISSN = {0303-1179,2492-5926},
      ISBN = {978-2-85629-283-9},
   MRCLASS = {58J65 (53C29 57M15 57R56 58D20 60J60)},
  MRNUMBER = {2667871},
MRREVIEWER = {Ambar\ N.\ Sengupta},
}

@misc{Lev26,
      title={A combinatorial formula for {W}ilson loop expectations on compact surfaces}, 
      author={L{\'{e}}vy, Thierry},
      year={2026},
      eprint={2603.08509},
      archivePrefix={arXiv},
      primaryClass={math.PR},
      url={https://arxiv.org/abs/2603.08509},
      note={arXiv:2603.08509},
}

@article {Liu96,
    AUTHOR = {Liu, Kefeng},
     TITLE = {Heat kernel and moduli space},
   JOURNAL = {Math. Res. Lett.},
  FJOURNAL = {Mathematical Research Letters},
    VOLUME = {3},
      YEAR = {1996},
    NUMBER = {6},
     PAGES = {743--762},
      ISSN = {1073-2780},
   MRCLASS = {58G11 (32G08 32G13)},
  MRNUMBER = {1426532},
MRREVIEWER = {Philippe\ P.\ Eyssidieux},
       DOI = {10.4310/MRL.1996.v3.n6.a3},
       URL = {https://doi.org/10.4310/MRL.1996.v3.n6.a3},
}

@article {OeckPfei01,
    AUTHOR = {Oeckl, Robert and Pfeiffer, Hendryk},
     TITLE = {The dual of pure non-abelian lattice gauge theory as a spin
              foam model},
   JOURNAL = {Nuclear Phys. B},
  FJOURNAL = {Nuclear Physics. B. Theoretical, Phenomenological, and
              Experimental High Energy Physics. Quantum Field Theory and
              Statistical Systems},
    VOLUME = {598},
      YEAR = {2001},
    NUMBER = {1-2},
     PAGES = {400--426},
      ISSN = {0550-3213,1873-1562},
   MRCLASS = {81T25 (83C27)},
  MRNUMBER = {1818475},
MRREVIEWER = {Daniele\ Oriti},
       DOI = {10.1016/S0550-3213(00)00770-7},
       URL = {https://doi.org/10.1016/S0550-3213(00)00770-7},
}

@article {Rus90,
    AUTHOR = {Rusakov, B. Ye.},
     TITLE = {Loop averages and partition functions in {${\rm U}(N)$} gauge
              theory on two-dimensional manifolds},
   JOURNAL = {Modern Phys. Lett. A},
  FJOURNAL = {Modern Physics Letters A. Particles and Fields, Gravitation,
              Cosmology, Astrophysics, Nuclear Physics, Accelerator Physics,
              Quantum Information},
    VOLUME = {5},
      YEAR = {1990},
    NUMBER = {9},
     PAGES = {693--703},
      ISSN = {0217-7323,1793-6632},
   MRCLASS = {81T13 (81T20 81T25)},
  MRNUMBER = {1051372},
MRREVIEWER = {E.\ Wieczorek},
       DOI = {10.1142/S0217732390000780},
       URL = {https://doi.org/10.1142/S0217732390000780},
}

@article {Sen97,
    AUTHOR = {Sengupta, Ambar},
     TITLE = {Gauge theory on compact surfaces},
   JOURNAL = {Mem. Amer. Math. Soc.},
  FJOURNAL = {Memoirs of the American Mathematical Society},
    VOLUME = {126},
      YEAR = {1997},
    NUMBER = {600},
     PAGES = {viii+85},
      ISSN = {0065-9266,1947-6221},
   MRCLASS = {58D20 (81S20 81T13)},
  MRNUMBER = {1346931},
MRREVIEWER = {Dana\ S.\ Fine},
       DOI = {10.1090/memo/0600},
       URL = {https://doi.org/10.1090/memo/0600},
}

@article {Sen97limit,
    AUTHOR = {Sengupta, Ambar},
     TITLE = {Yang-{M}ills on surfaces with boundary: quantum theory and
              symplectic limit},
   JOURNAL = {Comm. Math. Phys.},
  FJOURNAL = {Communications in Mathematical Physics},
    VOLUME = {183},
      YEAR = {1997},
    NUMBER = {3},
     PAGES = {661--705},
      ISSN = {0010-3616,1432-0916},
   MRCLASS = {58D20 (53C07 58D27 81T13)},
  MRNUMBER = {1462231},
MRREVIEWER = {Dana\ S.\ Fine},
       DOI = {10.1007/s002200050047},
       URL = {https://doi.org/10.1007/s002200050047},
}

@article {Wit91,
    AUTHOR = {Witten, Edward},
     TITLE = {On quantum gauge theories in two dimensions},
   JOURNAL = {Comm. Math. Phys.},
  FJOURNAL = {Communications in Mathematical Physics},
    VOLUME = {141},
      YEAR = {1991},
    NUMBER = {1},
     PAGES = {153--209},
      ISSN = {0010-3616,1432-0916},
   MRCLASS = {58G26 (14D20 32G13 58D27 81T13)},
  MRNUMBER = {1133264},
MRREVIEWER = {Dana\ S.\ Fine},
       URL = {http://projecteuclid.org/euclid.cmp/1104248198},
}

@book{Mun66,
  author    = {Munkres, James R.},
  title     = {Elementary Differential Topology},
  series    = {Annals of Mathematics Studies},
  volume    = {54},
  publisher = {Princeton University Press},
  address   = {Princeton, NJ},
  year      = {1966}
}

@book{RourkeSandersonPL,
  author    = {Rourke, C. P. and Sanderson, B. J.},
  title     = {Introduction to Piecewise-Linear Topology},
  edition   = {Revised reprint of the 1972 original},
  series    = {Ergebnisse der Mathematik und ihrer Grenzgebiete},
  volume    = {69},
  publisher = {Springer},
  year      = {1982},
  isbn      = {978-3-540-11102-3},
  doi       = {10.1007/978-3-642-81735-9}
}

\end{document}